\documentclass[12pt]{amsart}
\usepackage{graphicx}
\usepackage{epsfig}
\usepackage{fancyhdr}

\usepackage{amsmath,amsfonts,amssymb,amsthm}
\theoremstyle{plain}
\newtheorem{thm}{Theorem}
\newtheorem*{thm*}{Theorem}
\newtheorem*{Mthm}{Theorem \ref{main}}
\newtheorem*{Mthm2}{Theorem \ref{main2}}
\newtheorem{lem}{Lemma}
\newtheorem*{lem*}{Lemma}
\newtheorem{cor}{Corollary}
\newtheorem*{cor*}{Corollary}
\newtheorem*{Mcor}{Corollary \ref{main3}}
\newtheorem{prop}{Proposition}
\newtheorem*{prop*}{Proposition}
\newtheorem{obs}{Observation}

\theoremstyle{definition}

\newtheorem{Def}{Definition}
\newtheorem{Rem}{Remark}

\newcommand{\reftit}{\textit}    
\newcommand{\refis}{\textbf}     
\def\cR{{\mathcal{R}}}
\def\cN{{\mathcal{N}}}

\def\cT{{\mathcal{T}}}
\def\cP{{\mathcal{P}}}
\def\cQ{{\mathcal{Q}}}

\def\be{\begin{equation}}
\def\ee{\end{equation}}
\def\ben{\begin{equation*}}
\def\een{\end{equation*}}

\begin{document}

\title{Horton self-similarity of Kingman's coalescent tree}
\author{Yevgeniy Kovchegov}
\address{Department of Mathematics, Oregon State University, Corvallis, OR  97331}
\email{kovchegy@math.oregonstate.edu}
\thanks{YK was supported in part by a grant from the Simons Foundation (\#284262 to Yevgeniy Kovchegov) and by the NSF Award DMS 1412557.}

\author{Ilya Zaliapin}
\address{Department of Mathematics and Statistics, University of Nevada, Reno, NV, 89557-0084, USA}
\email{zal@unr.edu}
\thanks{IZ was supported by the NSF Awards DMS 0934871 and DMS 1049092.}



\begin{abstract}
The paper establishes Horton self-similarity for a tree representation of
Kingman's coalescent process.
The proof is based on a Smoluchowski-type system of ordinary 
differential equations that describes evolution of the number of 
branches of a given Horton-Strahler  
order in a tree that represents Kingman's $N$-coalescent, in a hydrodynamic limit.
We also demonstrate a close connection between the combinatorial Kingman's tree
and the combinatorial level set tree of a white noise, which implies Horton 
self-similarity for the latter.
\end{abstract}



\date{\today}
\maketitle

\section{Introduction}
\label{intro}
This study focuses on {\it Horton self-similarity} for binary rooted tree graphs.
The concept is related to Horton-Strahler ordering
of the tree branches \cite{Horton45, Strahler} that was introduced in hydrology in the mid-20th 
century to describe the dendritic structure of river networks and 
has penetrated other areas of sciences since then \cite{DK94,Vien90,BWW00}.
Devroye and Kruszewski \cite{DK94} assert that {\it ``the Horton-Strahler number occur
in almost every field involving some kind of natural branching pattern''}. 
Roughly speaking, the Horton-Strahler order corresponds to the relative importance of a branch in the tree hierarchy.
Specifically, each leaf is assigned order $k=1$; and
each internal vertex with offsprings of orders $i$ and $j$ is assigned
order $k=\max(i,j)+\delta_{ij}$, where $\delta_{ij}$ is the Kronecker's delta.
A {\it branch} is defined as a sequence of connected vertices with
the same order.

Horton self-similarity refers to the geometric decay of the number
$N_k$ of branches of order $k$ \cite{Horton45,Shreve66}.
A trivial example of Horton self-similarity is given by a perfect binary tree
(with all leaves having the same depth) for which $N_k/N_{k+1}=2$ for all
$1\le k <\Omega-1$, with $\Omega$ being the maximal branch order in the tree.
It is easily seen that for any non-perfect binary tree $N_k/N_{k+1}\ge2$,
with the strict inequality holding for at least one value of $k$.
A classical model that exhibits non-trivial Horton self-similarity is a tree 
representation of critical binary Galton-Watson branching processes 
\cite{BWW00,Pec95,Pitman},
also known in hydrology as Shreve's random topology model for river 
networks \cite{Shreve66,Shreve67}.
Ronald Shreve \cite{Shreve67} has demonstrated that in this model
the ratios $N_k/N_{k+1}$ converge to $R=4$ as $k$ increases.  
Recently, the authors established Horton self-similarity with the same 
asymptotic ratio for the level set tree representation of a homogeneous 
symmetric Markov chain and demonstrated that in general this representation
is not equivalent to the critical Galton-Watson tree \cite{ZK12}.
Models that obey Horton self-similarity with ratio 
different from $R=2,4$ are still lacking, however, despite their 
demonstrated practical importance 
\cite{BWW00,NTG97,Pec95,ZZF13}.

This study is a first step toward exploring Horton self-similarity 
with ratio $R\ne 2,4$.
We consider here the tree generated by Kingman's coalescent process with $N$
particles.
The main result is a weaker form of Horton self-similarity,
called here {\it root-Horton law}.
The Horton ratio is estimated numerically as $R=3.043827\hdots$.
We also establish a close relation between the combinatorial
tree representations of Kingman's $N$-coalescent and a 
combinatorial level set tree for a sequence of i.i.d. random 
variables (referred to as {\it discrete white noise}),
which implies Horton self-similarity for the latter.
These findings add two important classes of processes -- 
Kingman's coalescent and discrete white noise -- to the realm of
Horton self-similar systems. 

The paper is organized as follows. 
Section~\ref{self} describes Horton-Strahler ordering of tree branches
and the related concept of Horton self-similarity.
Kingman's coalescent process and its tree representation are defined 
in Sect.~\ref{coalescent}.
The main results are summarized in Sect.~\ref{results}.
Section~\ref{SHSODE} introduces the Smoluchowski-Horton system of equations 
that describes the dynamics of Horton-Strahler branches in Kingman's coalescent.
This section also establishes 
the validity of the Smoluchowski-Horton equations, as well as the existence 
of some related quantities, in hydrodynamic limit.
A proof of the existence of
root-Horton law for Kingman's coalescent is presented in Sect.~\ref{existence}.
Section~\ref{white} demonstrates a connection between
the combinatorial tree representation of Kingman's $N$-coalescent process and
combinatorial level set tree of a discrete white noise.
The Smoluchowski-Horton system for a general coalescent process with
collision kernel is written in Sect.~\ref{general}.
Section~\ref{discussion} concludes.  

\section{Self-similar trees}
\label{self}
This section defines {\it Horton self-similarity} for rooted
binary trees.

\subsection{Rooted trees}
A {\it graph} $\mathcal{G}=(V,E)$ is a collection of vertices
$V=\{v_i\}$, $1\le i \le N_V$ and edges 
$E=\{e_k\}$, $1\le k \le N_E$. 
In a {\it simple} undirected graph each edge is defined as an unordered
pair of distinct vertices: 
$\forall\, 1\le k \le N_E, \exists! \, 1\le i,j \le N_V, i\ne j$ 
such that $e_k=(v_i,v_j)$ and
we say that the edge $k$ {\it connects} vertices $v_i$ and $v_j$.
Furthermore, each pair of vertices in a simple graph may have 
at most one connecting edge.
A {\it tree} is a connected simple graph $T=(V,E)$ without cycles.
In a {\it rooted} tree, one node is designated as a root; this
imposes a natural {\it direction} of edges as well as the 
parent-child relationship between the vertices. 
Specifically, of the two connected vertices the one closest to
the root is called {\it parent}, and the other -- {\it child}.
Sometimes we consider trees embedded in a plane ({\it planar trees}), 
where the children of the same parent are ordered.

A {\it time oriented tree} $T=(V,E,S)$ assigns time marks $S=\{s_i\}$,
$1\le i \le N_V$ to the tree vertices in such a way that
the parent mark is always larger than that of its children.
A {\it combinatorial tree} $\textsc{shape}(T)\equiv(V,E)$ discards 
the time marks of a time oriented tree $T$, as well as possible 
planar embedding, and only preserves its graph-theoretic structure.

We often work with the space $\cT_N$ of combinatorial (not labeled, 
not embedded) rooted binary trees with $N$ leaves, and the space $\cT$
of all (finite or infinite) rooted binary trees.

\subsection{The Horton-Strahler orders}
\label{hst}
The Horton-Strahler ordering of the vertices of a finite 
rooted binary tree is performed in a hierarchical fashion, 
from leaves to the root \cite{Pec95,NTG97,BWW00}.
Specifically, each leaf has order $k({\rm leaf})=1$.
An internal vertex $p$ whose children have orders $i$ and $j$
is assigned the order 
\[k(p)=\max\left(i,j\right)+\delta_{ij},\] 
where $\delta_{ij}$ is the Kronecker's delta.
Figure~\ref{fig_HST} illustrates this definition.
A {\it branch} is defined as a union of connected vertices with 
the same order. 

\subsection{Horton self-similarity}
\label{sst}

Let $\cQ_N$ be a probability measure on $\cT_N$ and 
$N_k^{( \cQ_N)}$ be the number of branches of Horton-Strahler 
order $k$ in a tree generated according to $\cQ_N$. 

\begin{Def}
We say that a sequence of probability laws $\{ \cQ_N\}_{N \in \mathbb{N}}$ has {\it well-defined asymptotic Horton ratios} if for each $k \in \mathbb{N}^+$, random variables 
$\left(N_k^{( \cQ_N)}/N\right)$ converge in probability, as $N \rightarrow \infty$, to a constant value $\cN_k$, called the {\it asymptotic ratio} of the branches of order $k$. 
\end{Def} 

Horton self-similarity implies that the sequence $\cN_k$ decreases in a 
geometric fashion as $k$ goes to infinity. 
In this work we use a particular form of decay described below.

\begin{Def}
A sequence $\{ \cQ_N\}_{N \in \mathbb{N}}$ of probability laws on $\cT$ with well-defined asymptotic Horton ratios is said to obey a {\it root-Horton self-similarity law} if and only if 
the following limit exists and is finite and positive:
$\lim\limits_{k \rightarrow \infty} 
\Big( \cN_k \Big)^{-{1 \over k}}=R>0.$
The constant $R$ is called the {\it Horton exponent}.
\end{Def}

\section{Coalescent processes, trees}
\label{coalescent}
This section reviews Kingman's coalescent 
process with $N$ particles and introduces 
its tree representation. 

\subsection{Kingman's $N$-coalescent process}
We start by considering a general finite coalescent process defined by a collision kernel  \cite{Bertoin,Pitman,Berestycki}. 
The process begins with $N$ particles (clusters) of mass one. 
The cluster formation is governed by a symmetric collision rate 
kernel $K(i,j)=K(j,i)>0$.
Namely, a pair of clusters with masses $i$ and $j$ coalesces at the 
rate $K(i,j)$, independently of the other pairs, to 
form a new cluster of mass $i+j$.
The process continues until there is a single cluster of mass $N$.

\medskip
\noindent
Formally, for a given $N$ consider the space $\cP_{[N]}$ of 
partitions of $[N]=\{1,2,\hdots,N\}$. 
Let $\Pi^{(N)}_0$ be the initial partition in singletons, 
and $\Pi^{(N)}_t ~~(t \geq 0)$ be a strong Markov process such that 
$\Pi^{(N)}_t$ transitions from partition $\pi \in \cP_{[N]}$ to 
$\pi' \in \cP_{[N]}$ with rate $K(i,j)$ provided that partition 
$\pi'$ is obtained from partition $\pi$ by merging two clusters 
of $\pi$ of masses $i$ and $j$.
If $K(i,j) \equiv 1$ for all positive integer masses $i$ and $j$, 
the process $\Pi^{(N)}_t$ is known as Kingman's $N$-coalescent process.

\subsection{Coalescent tree}
A merger history of Kingman's $N$-coalescent process can 
be naturally described by a time oriented binary tree $T^{(N)}_{\rm K}$ constructed as follows.
Start with $N$ leaves that represent the initial $N$ particles and have time mark $t=0$. 
When two clusters coalesce (a transition occurs), merge the corresponding vertices 
to form an internal vertex with a time mark of the coalescent.
The final coalescence forms the tree root.
The resulting time oriented binary tree represents the history of the process.
We notice that a given unlabeled tree corresponds to multiple coalescent trajectories 
obtained by relabeling of the initial particles.

Observe that the combinatorial version $\textsc{shape}\left(T_{\rm K}^{(N)}\right)$ 
of the Kingman's coalescent tree 
is invariant under time scaling $t_{\rm new}=C\,t_{\rm old}$, $C>0$.  
Thus without loss of generality we let $K(i,j) \equiv 1/N$ in Kingman's 
$N$-coalescent process. 
Slowing the process's evolution $N$ times is natural in Smoluchowski coagulation 
equations that describe the dynamics of the fraction of clusters of different 
masses.

\section{Statement of results}
\label{results}
The main result of this paper is root-Horton self-similarity for the 
combinatorial tree $\textsc{shape}\left(T^{(N)}_{\rm K}\right)$ of the 
Kingman's $N$-coalescent process, as $N$ goes to infinity.
Specifically, let $N_k$ denote the number of branches of Horton-Strahler order $k$ in 
the tree $T^{(N)}_{\rm K}$ that describes Kingman $N$-coalescent. 
We show in Sect.~\ref{SHSODE}, Lemma~\ref{lem1} that for each $k\ge 1$, 
$N_k/N$ converges in probability to the asymptotic  Horton ratio 
\[\cN_k=\lim_{N\to\infty} N_k/N.\]  
Moreover, these $\cN_k$ are finite and can be expressed as
\[\cN_k = \frac{1}{2}\int_0^\infty g_k^2(x)\,dx,\]
where the sequence $g_k(x)$ 
solves the following system of
ordinary differential equations (ODEs):
\[g'_{k+1}(x)-{g^2_k(x) \over 2}+g_k(x) g_{k+1}(x)=0,\quad x\ge 0\] 
with $g_1(x)=2/(x+2)$, $g_k(0)=0$ for $k \geq 2$.
Equivalently,
\[\cN_k=\int_0^1 \left(1-\left(1-x\right)h_{k-1}(x)\right)^2 dx,\]
where $h_0\equiv 0$ and the sequence $h_k(x)$  satisfies the ODE system
\[h'_{k+1}(x)=2h_k(x)h_{k+1}(x)-h_k^2(x),\quad 0\le x \le 1\]
with the initial conditions $h_k(0)=1$ for $k\ge 1$.

The root-law Horton self-similarity is proven in Section \ref{existence}
in the following statement.  
\begin{Mthm}
The asymptotic Horton ratios $\cN_k$ exist and finite and satisfy
the convergence 
$\lim\limits_{k \rightarrow \infty}\left(\cN_k \right)^{-{1 \over k}}=R$
with $2 \le R \le 4$.
\end{Mthm}

Numerical solution for the sequence $h_k$ provides an estimation 
of Horton exponent $R=3.043827\hdots$ and suggests that $\cN_k$ 
also obey a stronger version of Horton self-similarity:
$ \lim\limits_{k \rightarrow \infty} \left(\cN_k \, R^k\right) = N_0>0$.

Section~\ref{level} introduces a {\it level set tree} $\textsc{level}(X_i)$ that describes 
the structure of the level sets of a discrete-time function $X_i$, $i=1,\dots,i_{\rm max}$.
In particular, we show that there exists a one-to-one map between finite rooted planar 
time oriented binary trees and sequences of the local extrema of $X_i$.
Let $W=\{W_i\}$ be a {\it discrete white noise}, that is a process
comprised of i.i.d. values with a common atomless distribution. 
Consider now a process $\tilde W^{(N)}_i$ with exactly $N$ local maxima separated by 
$N-1$ internal local minima such that the latter form a discrete white noise; 
we call $\tilde W^{(N)}_i$ an {\it extended discrete white noise}.

Let $L^{(N)}_W=\textsc{level}\left(\tilde W^{(N)}_i\right)$ be the 
level set tree of $\tilde W^{(N)}_i$ and 
$\textsc{shape}\left(L^{(N)}_W\right)$ be the combinatorial tree 
that retains the graph-theoretic structure of
$L^{(N)}_W$ and drops its planar embedding as well as the 
time marks of the vertices.
Furthermore, let $T^{(N)}_{\rm K}$ be the tree that corresponds to a 
Kingman's $N$-coalescent, and let
$\textsc{shape}\left(T^{(N)}_{\rm K}\right)$ be its combinatorial version that
drops the time marks of the vertices.
By construction, both the trees $\textsc{shape}\left(L^{(N)}_W\right)$ and 
$\textsc{shape}\left(T^{(N)}_{\rm K}\right)$, belong to the space $\cT_N$ of 
binary rooted trees with $N$ leaves.
Section~\ref{finite} establishes the following equivalence. 

\begin{Mthm2}\label{main2}
The trees $\textsc{shape}\left(L^{(N)}_W\right)$ and 
$\textsc{shape}\left(T^{(N)}_{\rm K}\right)$ have the same distribution on $\cT_N$.
\end{Mthm2} 

The equivalence leads to the Horton self-similarity for discrete white noise.
\begin{Mcor}
The combinatorial level set tree of a discrete white noise 
is root-Horton self similar with the same Horton 
exponent $R$ as for Kingman's coalescent.
\end{Mcor}

\section{Smoluchowski-Horton ODEs for Kingman's coalescent}  
\label{SHSODE}
Consider Kingman's $N$-coalescent process and its tree representation
$T^{(N)}_{\rm K}$.
In Section~\ref{informal} we informally write Smoluchowski-type 
ODEs for the number of Horton-Strahler branches in the coalescent 
tree $T^{(N)}_{\rm K}$ and consider the asymptotic version of these 
equations as $N\to\infty$.
Section~\ref{hydro} formally establishes the validity of the hydrodynamic limit.

\subsection{Main equation} 
\label{informal}
Recall that we let $K(i,j) \equiv 1/N$ in Kingman's $N$-coalescent process. Let $|\Pi^{(N)}_t|$ denote the total number of clusters at time $t \geq 0$, and let  $\eta_{(N)}(t):=|\Pi^{(N)}_t|/N$ be the total number of clusters relative to the system size $N$.
Then $\eta_{(N)}(0)=N/N=1$ and $\eta_{(N)}(t)$ decreases by $1/N$ with each coalescence of clusters with 
the rate
$${1 \over N} \, \binom{N\, \eta_{(N)}(t)}{2}={\eta_{(N)}^2(t) \over 2}\cdot N+o(N),
\quad{\rm as~} N\to\infty,$$
since $1/N$ is the coalescence rate for any pair of clusters regardless of their masses. 
Informally, this implies that the limit relative number of clusters 
$\displaystyle\eta(t)=\lim_{N\to\infty}\eta_{(N)}(t)$ 
satisfies the following ODE:
\begin{eqnarray} \label{Aeta_t}
{d \over dt} \eta(t)=-\frac{\eta^2(t)}{2}.
\end{eqnarray}
The corresponding initial condition $\eta(0)=1$ implies a unique solution $\eta(t)=2/(2+t)$.

Next, for any $k \in \mathbb{N}^+$ we define $\eta_{k,N}(t)$ to be the number of 
clusters that correspond to branches of Horton-Strahler order  $k$ at time $t$ 
relative to the system size $N$. 
Initially, each particle represents a leaf of Horton-Strahler order $1$. 
Accordingly, the initial conditions are set to be, using Kronecker's delta notation, 
$$\eta_{k,N}(0)=\delta_1(k).$$ 
We describe now the evolution of $\eta_{k,N}(t)$ using the definition of Horton-Strahler orders.

Observe that $~\eta_{k,N}(t)$ increases by $1/N$ with each coalescence of clusters of Horton-Strahler 
order $k-1$ that happens with the rate
$${1 \over N} \, 
\binom{N\, \eta_{k-1,N}(t)}{2}={\eta_{k-1,N}^2(t) \over 2} \cdot N+o(N).$$
Thus ${\eta_{k-1,N}^2(t) \over 2}+o(1)$ is the instantaneous rate of increase of $\eta_{k,N}(t)$.

Similarly,  $~\eta_{k,N}(t)$ decreases by $1/N$ when a cluster of order $k$ coalesces 
with a cluster of order strictly higher than $k$ with the rate
$$\eta_{k,N}(t) \, \left(\eta_{(N)}(t)-\sum\limits_{j=1}^{k} \eta_{j,N}(t) \right)\cdot N,$$
and it decreases by $2/N$ when a cluster 
of order $k$ coalesces with another cluster of order $k$ with the rate
$${1 \over N} \, \binom{N\, \eta_{k,N}(t)}{2} ={\eta_{k,N}^2(t) \over 2}\cdot N+o(N).$$
Thus the instantaneous rate of decrease of $\eta_{k,N}(t)$ is
$$\eta_{k,N}(t) \, \left(\eta_{(N)}(t)-\sum\limits_{j=1}^{k} \eta_{j,N}(t) \right)
+\eta^2_{k,N}(t)+o(1).$$

Now we can informally write the limit rates-in and the rates-out for the clusters 
of Horton-Strahler order via the following {\it Smoluchowski-Horton system} of ODEs:
\begin{eqnarray} \label{Aeta}
{d \over dt} \eta_k(t)=\frac{\eta^2_{k-1}(t)}{2}-\eta_k(t) 
\, \left(\eta(t)-\sum\limits_{j=1}^{k-1} \eta_j(t) \right)
\end{eqnarray}
with the initial conditions $\eta_k(0)=\delta_1(k)$. 
Here we define $\displaystyle\eta_k(t)=\lim_{N\to\infty}\eta_{k,N}(t)$, provided it exists, and 
let $\eta_0 \equiv 0$.

Since $\eta_k(t)$ has the instantaneous rate of increase ${\eta_{k-1}^2(t) \over 2}$, 
the relative total number of clusters corresponding to branches of Horton-Strahler order  $k$ is given by
\be
\label{cNj}
\cN_k=\delta_1(k)+\int\limits_0^{\infty} {\eta^2_{k-1}(t) \over 2} dt.
\ee  

It is not hard to compute the first three terms of the sequence 
$\cN_k$ by solving equations (\ref{Aeta_t}) and (\ref{Aeta}) 
in the first three iterations:
$$\cN_1=1, \quad \cN_2={1 \over 3}, \quad \text{ and } 
\quad \cN_3={e^4 \over 128}-{e^2 \over 8}+{233 \over  384}= 0.109686868100941\hdots$$
Hence, we have 
${\cN_1/ \cN_2}={3}$ and 
${\cN_2/ \cN_3}= 3.038953879388\dots$ 
Our numerical results yield, moreover,
$$\lim\limits_{k \rightarrow \infty}\left(\cN_k \right)^{-{1 \over k}}
=\lim\limits_{k \rightarrow \infty}{\cN_{k} \over \cN_{k+1}}=3.0438279\dots$$

\subsection{Hydrodynamic limit} 
\label{hydro}
This section establishes the existence of the asymptotic ratios 
$\cN_k$ as well as the validity of the equations~\eqref{Aeta_t},
\eqref{Aeta} and \eqref{cNj} in a hydrodynamic limit. 
We refer to Darling and Norris \cite{RDJN08} for a survey of formal techniques 
for proving that a Markov chain converges to the solution of a differential equation.

Notice that {\it quasilinearity} of the 
system of ODEs in (\ref{Aeta}) implies the existence and uniqueness. 
Specifically, if the first $k-1$ functions $\eta_1(t),\hdots,\eta_{k-1}(t)$ 
are given, then (\ref{Aeta}) is a linear equation in $\eta_k(t)$. 
The following argument is different from the one presented by Norris \cite{Norris99} for the Smoluchowski equations. 

\begin{lem}
\label{lem2}
Let $\eta_{(N)}(t)$ be the relative total number of clusters and $\eta(t)$  be the solution to equation (\ref{Aeta_t}) with the initial 
condition $\eta(0)=1$. Then
\[\big\|\eta_{(N)}(t)-\eta (t) \big\|_{L^\infty [0,\infty)} \rightarrow 0\]
in probability as $N\to\infty$.
\end{lem}

A proof of Lemma \ref{lem2} is given in Appendix \ref{AA2}. 
The proof is divided into steps that we briefly outline below.

\bigskip
\noindent
$\bullet$ \textit{Steps I, II.}  
We start by establishing bounds on the number of coalescences 
within the time interval $[t,t+\delta]$.
Specifically, fix $\epsilon_0 \in (0,1)$ and take $\delta>0$.
Given $y \in {1 \over N} \mathbb{Z} \cap [\epsilon_0,1]$,
let $u=\binom{N y}{2}$ and $v=\binom{N y-\lceil {\delta y^2 \over 2}N \rceil}{2}$.
We use the exponential Markov inequality to show that for any given $t \geq 0$
and large enough $N$ we have
\[P\left( {\delta \over N^2}v-(1+\delta)N^{-1/3} \leq \eta_{(N)}(t)-\eta_{(N)}(t+\delta) \leq {\delta \over N^2} u+N^{-1/3}~~\Big| ~\eta_{(N)}(t)=y~\right)\]
\[\geq  \left(1-\exp\left\{-N^{1/6}+{4\delta \over \epsilon_0^2} \right\} \right)^2.\]

\bigskip
$\bullet$ \textit{Step III.} 
The bounds of steps I, II are applied to show that\\

$P\left( ~\left|{\eta^2_{(N)}(t) \over 2} +\Delta_\delta \eta_{(N)}(t) \right| ~\leq ~\delta+(\delta^{-1}+1)N^{-1/3}~~\Big| ~\eta_{(N)}(t)=y~\right)$
\begin{equation} \label{ineq:probdiffs}
\geq \left(1-\exp\left\{-N^{1/6}+{4\delta \over \epsilon_0^2} \right\} \right)^2 
\end{equation}
for $N$ large enough, where  $~\Delta_\delta f(x):={f(x+\delta)-f(x) \over \delta}~$ denotes the forward difference.

\bigskip
$\bullet$ \textit{Step IV.} For $K>0$, consider an interval $[0,K]$ partitioned into $M$ subintervals 
$$[t_0,t_1],~[t_1,t_2], ~\hdots, ~[t_{M-1},t_M]$$
of equal length $\delta=K/M$, where $t_0=0$ and $t_M=K$. 
Let $\epsilon_0=\eta(K)/2=1/(2+K)$, where $\eta(t)=2/(2+t)$ is the solution to the equation (\ref{Aeta_t}) with the initial 
condition $\eta(0)=1$. 
Consider the difference equation
\begin{equation}\label{psiEs}
\Delta_\delta \psi_{(N)}(t_i)=-{\psi_{(N)}^2(t_i) \over 2}+\mathcal{E}'(t_i)
\end{equation}
with initial condition $\psi_{(N)}(0)=1$, where the error $~|\mathcal{E}'(t_i)| \leq \delta+(\delta^{-1}+1)N^{-1/3}$.  
At this step we prove that if $M$ is large enough and for any natural number $j \leq M$ function $\psi_{(N)}(t_i)$ satisfies (\ref{psiEs}) for all $i \in \{0,1,\hdots,j-1\}$, then 
$$\psi_{(N)}(t_j) \geq \epsilon_0$$ as we take $N$ large enough. This follows from observing that $\eta(t)$ will satisfy a difference equation similar to (\ref{psiEs}),
\begin{equation}\label{etaEs}
\Delta_\delta \eta(t_i)=-{\eta^2(t_i) \over 2}+\mathcal{E}(t_i)
\end{equation}
with $~|\mathcal{E}(t_i)| \leq {1 \over 4} \delta~$ for all $i \in \{0,1,\hdots,M-1\}$.

\bigskip
$\bullet$ \textit{Step V.} 
Consider events
\begin{equation}\label{defAis}
A_i=\left\{~\Delta_\delta \eta_{(N)}(t_i)=-{\eta^2_{(N)}(t_i) \over 2}+\mathcal{E}'(t_i) ~\text{ and } ~|\mathcal{E}'(t_i)| ~\leq \delta+(\delta^{-1}+1)N^{-1/3}~\right\}
\end{equation}
for all $i \in \{0,1,\hdots,M-1\}$. 
Here we combine the results of steps III and IV and establish that with probability greater than 
$P\left( ~\bigcap\limits_{i=0}^{M-1}A_i~\right) \rightarrow 1$ as $M \rightarrow \infty$, 
$~\eta_{(N)}(t_i)$ satisfies the difference equation (\ref{psiEs}) with $\psi_{(N)}(t) \equiv \eta_{(N)}(t)$.

\bigskip
$\bullet$ \textit{Step VI.} 
Taking $\psi_{(N)}(t) \equiv \eta_{(N)}(t)$, we compare the difference equation (\ref{psiEs}) with (\ref{etaEs}), and bound the error $|\eta_{(N)}(t)-\eta(t)|$ for all $t \in [0,K]$. 
Specifically, we show that with probability greater than 
$~P\left( ~\bigcap\limits_{i=0}^{M-1}A_i~\right) \rightarrow 1$,
\begin{equation} \label{ineq:KMs}
\big\|\eta_{(N)}(t)-\eta (t) \big\|_{L^\infty [0,K]} \leq
{15 \over 4}K^2/M+4K/M+3/M
\end{equation}
for $M$ large enough and $N \geq M^6$. 
Therefore, letting $M \rightarrow \infty$, we obtain
\[\big\|\eta_{(N)}(t)-\eta (t) \big\|_{L^\infty [0,K]} \rightarrow 0 \qquad 
\text{ in probability.}\]

\bigskip
$\bullet$ \textit{Step VII.} 
Take $\epsilon \in (0,1)$ and $\gamma>1$, and consider $K>{2(1-\epsilon) \over \epsilon}\gamma$.
This step uses Markov inequality to show that
$$P\Big(\big\|\eta_{(N)}(t)-\eta (t) \big\|_{L^\infty [K,\infty)}<\epsilon \Big) ~> 1- 1/\gamma,$$
which, together with the results of step VI, implies
$$\limsup\limits_{N \rightarrow \infty}P\Big(\big\|\eta_{(N)}(t)-\eta (t) \big\|_{L^\infty [0,\infty)}<\epsilon \Big) \geq 1-1/\gamma.$$
We conclude that
$$\lim\limits_{N \rightarrow \infty}P\Big(\big\|\eta_{(N)}(t)-\eta (t) \big\|_{L^\infty [0,\infty)}<\epsilon \Big)=1.$$

\medskip
\noindent
Therefore we have shown that $~\|\eta_{(N)}(t)-\eta(t)  \|_{L^\infty [0,\infty)} \rightarrow 0~$ in probability, thus establishing Lemma \ref{lem2}.

\vspace{1cm}

We now proceed with establishing a hydrodynamic limit for the 
Smoluchowski-Horton system of ODEs \eqref{Aeta}.  
Let $$\eta_{k,N}(t):={N_k(t) \over N}~~~\text{ and }~~~g_{k,N}(t):=\eta_{(N)}(t)-\sum\limits_{j:j<k}\eta_{j,N}(t).$$ 

\begin{lem}
\label{lem3}
Consider the relative numbers $\eta_{k,N}(t)$ of clusters that correspond to 
branches of Horton-Strahler order $k$ and functions $\eta_k(t)$ that solve the 
system of equations \eqref{Aeta} with the initial conditions $\eta_k(0)=\delta_1(k)$. 
Then, 
\[\|\eta_{k,N}(t)-\eta_k(t)  \|_{L^\infty [0,\infty)} \rightarrow 0,\quad \forall k\ge1, \] 
in probability, as $N\to\infty$.
\end{lem}

A proof of Lemma \ref{lem3} is given in Appendix \ref{AA3}. Here we summarize the steps used in the proof.

\medskip
$\bullet$ \textit{Step I.} 
We use the setting from the proof of Lemma \ref{lem2}. 
Fix $K>0$ and consider an interval $[0,K]$ partitioned into $M$ subintervals 
$$[t_0,t_1],~[t_1,t_2], ~\hdots,~[t_{M-1},t_M]$$
of equal length $\delta=K/M$, where $t_0=0$ and $t_M=K$. 
Let $\epsilon_0=\eta(K)/2=1/(2+K)$.
The total number of coalescences within the 
interval $[t_i,t_{i+1}]$ equals $N\big[\eta_{(N)}(t_i)-\eta_{(N)}(t_{i+1})\big]$.

For any $k \in \mathbb{N}^+$ and any $i=0,1,\hdots,M-1$ we represent the relative number of coalescences that involve the clusters of order $k$ within $[t_i,t_{i+1}]$ as 
$$\eta_{k,N}(t_{i+1})-\eta_{k,N}(t_i)=\xi_1+\xi_2+\hdots+\xi_{m_i},$$
where $~\xi_1,\xi_2,\hdots ,\xi_{m_i}~$ are random variables that correspond to 
the $m_i$ coalescences (of any Horton-Strahler order) within $[t_i,t_{i+1}]$ in 
the order of occurrence.
Here, each $\xi_r$ can take values in $\frac{1}{N}\{-2,-1,0,1\}$; and their dependence 
on $k$ is omitted to simplify the notations.
By construction, conditioned on the values $\{\eta_{j,N}(t_i)\}_j$, the distribution of $\xi_r$ for $1\le r\le m_i$ is completely determined 
by the history $~\mathcal{T}_{r-1}$ of the preceding $r-1$ transitions.

Consider a random variable $\xi$ with the values $\{-2,-1,0,1\}$ specified by the 
probabilities $\{p(-2),p(-1),p(0),p(1)\}$: 
\begin{eqnarray*}
p(-2)&:=& \eta^2_{k,N}(t_i) / \eta^2_{(N)}(t_i),\\ 
p(1) &:=& 
\begin{cases}
    \eta^2_{k-1,N}(t_i) / \eta^2_{(N)}(t_i) & \text{ if } k>1 \\
   0   & \text{ if } k=1
\end{cases},\\
p(-1) &:=& 
2\eta_{k,N}(t_i)g_{k+1,N}(t_i)/  \eta^2_{(N)}(t_i),\\
p(0) &:=& 1-p(-2)-p(-1)-p(1).
\end{eqnarray*}
Recall the events $A_i$ defined in (\ref{defAis}). We notice that, conditioned on $\bigcap\limits_{i'=0}^{i}A_{i'}$, the total variation distance between the distribution of $\xi_r $ (for a fixed $1\le r\le m_i$) and the distribution of $\xi$ is of order $\mathcal{O}(\delta)$.
We use this to show that for each $k\in \mathbb{N}^+$, there is a large enough $c_k>0$ and $a>0$ such that
\begin{eqnarray}
\label{ineq:aMbs}
P\Big( \Big|\big[\eta_{k,N}(t_{i+1})-\eta_{k,N}(t_i) \big] &-& 
E[\xi]\delta {\eta^2_{(N)}(t_i) \over 2} \Big|
< c_k \delta^{4/3} ~\Big| ~\bigcap\limits_{i'=0}^{i}A_{i'}~ \Big)\nonumber\\
&\geq&   1-\exp\Big\{-a M^4\Big\}
\end{eqnarray}
for all $i=0,1,\hdots,M-1$, $~2M^6>N>M^6$, and $M$ large enough.

\bigskip
\noindent
$\bullet$ \textit{Step II.} 
According to the results of step I, we obtain the following system of difference equations:
\begin{eqnarray}\label{eqn:hydrs}
\Delta_\delta \eta_{1,N}(t_i) & = & -\eta_{1,N}(t_i)\eta_{(N)}(t_i) +\mathcal{E}'_1(t_i) \nonumber \\
& & \\
\Delta_\delta \eta_{k,N}(t_i) & = & {\eta^2_{k-1,N}(t_i) \over 2}  -\eta_{k,N}(t_i) g_{k,N}(t_i) +\mathcal{E}'_{k}(t_i)  \quad \text{ for } k \geq 2\nonumber 
\end{eqnarray}
with the initial conditions
 $$\Big(\eta_{1,N}(0), ~\eta_{2,N}(0), ~\hdots, ~\eta_{k,N}(0), ~\hdots \Big)=(1,0,0,\hdots),$$
where for a given $\rho\in \mathbb{N}$ and $c=\max\limits_{1\le k\le \rho}\{c_k\}$ we have $|\mathcal{E}'_k(t_i)| < c \delta^{1/3}$ for each $1\le k \le \rho$.
Each equation in this system holds with the probability that converges to unity as $M$ increases.

\medskip
\noindent
We now compare the above difference equations (\ref{eqn:hydrs}) to the following system of difference equations 
that corresponds to the system of ODEs (\ref{Aeta}):
\begin{eqnarray}\label{eqn:hydr2s}
\Delta_\delta \eta_1(t_i) & = & -\eta_1(t_i) \eta(t_i) +\mathcal{E}_1(t_i) \nonumber \\
& & \\
\Delta_\delta \eta_k(t_i) & = & {\eta^2_{k-1}(t_i) \over 2}  -\eta_k(t_i) g_k(t_i) +\mathcal{E}_{k}(t_i) \quad \text{ for } k \geq 2, \nonumber 
\end{eqnarray}
where $~g_k(t):=\eta(t)-\sum\limits_{i:~i<k} \eta_i(t)$, and the error
$$\mathcal{E}_{k}(t_i)={\eta''_k(c_{i,k}) \over 2}\delta \qquad \text{ for some } c_{i,k} \in (t_i,t_{i+1}).$$


\bigskip
\noindent
$\bullet$ \textit{Step III.} 
We show that, conditioning on the event $~\bigcap\limits_{i=0}^{M-1}A_i$, we have the following upper 
bound for any $k \in \{1,\hdots,\rho\}$, all $i \in \{0,1,\hdots,M-1\}$, and
$t \in (t_i,t_{i+1})$:
\begin{eqnarray*}
\big| \eta_{k,N}(t)-\eta_k(t) \big| & \leq & \big| \eta_{k,N}(t) - \eta_{k,N}(t_i)\big|+\big| \eta_{k,N}(t_i)-\eta_k(t_i) \big|+  \big|\eta_k (t_i)-\eta_k (t) \big| \\
& \leq & \left(5K^2+4K+4 \right)/M+(c+1)2^k{\delta^{1/3} \over \rho}\big[e^{2K\rho}-1 \big]+3\delta. 
\end{eqnarray*}

\medskip
\noindent
We conclude that, for any $k$,
$$\|\eta_{k,N}-\eta_k  \|_{L^\infty [0,K]} \rightarrow 0 \quad\text{in probability}.$$

\bigskip
\noindent
$\bullet$ \textit{Step IV.} Finally, observe that  for any $\epsilon>0$ and for $K>2$ large enough so that $~\eta(K) <\epsilon$,
$$\eta_k (t) \leq \eta(t) \leq \eta(K) <\epsilon \text{ for all } t \geq K$$
and 
\begin{eqnarray*}
P\Big(\big\|\eta_{k,N}(t)-\eta_k (t) \big\|_{L^\infty [K,\infty)}>\epsilon \Big) & \leq & P\Big(\big\|\eta_{k,N}(t) \big\|_{L^\infty [K,\infty)}>\epsilon \Big)\\
& \leq & P\Big(\big\|\eta_{(N)}(t) \big\|_{L^\infty [K,\infty)}>\epsilon \Big)\\
& = & P\Big(\eta_{(N)}(K)>\epsilon \Big) \\
& \leq & {2(1-\epsilon) \over \epsilon K}.
\end{eqnarray*}
The last bound is obtained from Markov inequality for the random variable $T_m$ 
that represents the time of the $m$-th coalescence.
Therefore, together with the result of the previous step, we have shown that 
for each $k$, 
$$\|\eta_{k,N}-\eta_k  \|_{L^\infty [0,\infty)} \rightarrow 0$$ in probability. 
This completes the proof.

\vspace{1cm}
Finally, the last lemma in this section establishes a hydrodynamic limit for the Horton ratios. 
\begin{lem}
\label{lem1}
The Horton ratios $~N_k/N~$  converge in probability to a finite constant $\cN_k$
given by \eqref{cNj}, as $~N \rightarrow \infty$.
\end{lem}

A proof of Lemma \ref{lem1} is given in Appendix \ref{AA1}.

\section{The root-Horton self-similarity and related results} 
\label{existence}
We begin this section with preliminary lemmas and propositions, 
and then proceed to proving Theorem \ref{main}.  

Let $g_1(t)=\eta(t)$ and $g_k(t)=\eta(t)-\sum\limits_{j:~j<k} \eta_j(t)$ 
be the asymptotic number of clusters of Horton order $k$ or higher at time $t$. 
We can rewrite (\ref{Aeta}) via $g_k$ using $\eta_k(t)=g_k(t)-g_{k+1}(t)$:
\[{d \over dt}g_k(t)-{d \over dt}g_{k+1}(t)
={\big(g_{k-1}(t)-g_k(t) \big)^2 \over 2}-(g_k(t)-g_{k+1}(t))g_k(t).\]
Observe that
$g_1(t) \geq g_2(t) \geq g_3(t) \geq \hdots$.
We now rearrange the terms, obtaining for all $k \geq 2$,
\begin{eqnarray} \label{odeF1}
 ~~~~~~{d \over dt}g_{k+1}(t)-{g^2_k(t) \over 2}+g_k(t) g_{k+1}(t)
 ={d \over dt} g_k(t)-{g^2_{k-1}(t) \over 2}+g_{k-1}(t) g_k(t).
\end{eqnarray}
One can readily check that ${d \over dt} g_2(t)-{g^2_1(t) \over 2}+g_1(t)\,g_2(t)=0$; 
the above equations hence simplify as follows
\begin{eqnarray} \label{odeG}
g'_{k+1}(t)-{g^2_k(t) \over 2}+g_k(t) g_{k+1}(t)=0 \qquad \\ \nonumber \text{ with }~g_1(t)={2 \over t+2}, \text{ and } g_k(0)=0 \text{ for } k \geq 2.
\end{eqnarray}

Next, returning to the asymptotic ratios of the number of order-$k$ branches to $N$, 
we observe that (\ref{odeF1}) implies that, for $k\ge 2$,
$$\cN_k=\int\limits_0^{\infty} {\eta^2_{k-1}(t) \over 2} dt
=\int\limits_0^{\infty} {(g_{k-1}(t)-g_k(t))^2 \over 2} dt
=\int\limits_0^{\infty} {g^2_k(t) \over 2} dt$$
since
$$ {(g_{k-1}(t)-g_k(t))^2 \over 2} ={d \over dt}g_k(t)+{g^2_k(t) \over 2},$$ 
where $\int\limits_0^{\infty}{d \over dt}g_k(t)dt=g_k(\infty)-g_k(0)=0$ for $k \geq 2$.
Let $n_k$ represent the number of order-$k$ branches relative to the number 
of order-$(k+1)$ branches:
$$n_k:={\cN_k \over \cN_{k+1}}={\frac{1}{2}\int\limits_0^{\infty} {g^2_k(t)} dt \over \frac{1}{2}\int\limits_0^{\infty} {g^2_{k+1}(t)} dt}={ \|g_k\|^2_{L^2[0, \infty)} \over  \|g_{k+1}\|^2_{L^2[0, \infty)}}.$$ 
Consider the following limits that represent respectively the 
root and the ratio asymptotic Horton laws:
$$\lim\limits_{k \rightarrow \infty}\left(\cN_k \right)^{-{1 \over k}}=\lim\limits_{k \rightarrow \infty}\left(\prod\limits_{j=1}^{k} n_j \right)^{-{1 \over k}} \qquad \text{ and } \qquad \lim\limits_{k \rightarrow \infty} n_k=\lim\limits_{k \rightarrow \infty} { \|g_k\|^2_{L^2[0, \infty)} \over  \|g_{k+1}\|^2_{L^2[0, \infty)}}.$$
Theorem~\ref{main} establishes the existence of the first limit. 
We expect the second, stronger, limit also to exist and both of them to be equal to $3.043827\dots$ 
according to our numerical results.
We now establish some basic facts about $g_k$ and $n_k$.
\begin{prop}  \label{prop1}
Let $g_k(x)$ solve the ODE system (\ref{odeG}). Then\\
\begin{description}
\item[\qquad (a)] 
$~\frac{1}{2}\int\limits_0^{\infty} {g^2_k(t)} dt=\int\limits_0^{\infty} g_k(t)g_{k+1}(t) dt,$\\
\item[\qquad (b)] 
$~\int\limits_0^{\infty} g^2_{k+1}(t) dt=\int\limits_0^{\infty} (g_k(t)-g_{k+1}(t))^2 dt,$\\
\item[\qquad (c)] 
$~\lim\limits_{t \rightarrow \infty} tg_k(t) =2,$\\
\item[\qquad (d)] 
$~n_k={ \|g_k\|^2_{L^2[0, \infty)} \over  \|g_{k+1}\|^2_{L^2[0, \infty)}} \geq {2},$\\
\item[\qquad (e)] 
$~n_k={ \|g_k\|^2_{L^2[0, \infty)} \over  \|g_{k+1}\|^2_{L^2[0, \infty)}} \leq {4}.$\\
\end{description}

\end{prop}

\begin{proof} 

Part (a) follows from integrating (\ref{odeG}), and part (b) follows from part (a). 
Part (c) is done by induction, using the L'H\^{o}pital's rule as follows. 
It is obvious that $~\lim\limits_{x \rightarrow \infty} tg_1(t) =2$. 
We observed earlier that $g_1(t) \geq g_2(t) \geq g_3(t) \geq \hdots$. Hence, for any $k \geq 1$,
$$tg_k(t) \leq tg_1(t)={2t \over t+2} <2 \quad \forall t \geq 0.$$
Also, $$[tg_{k+1}]'={tg^2_k(t) \over 2}-tg_k(t)g_{k+1}(t)+g_{k+1}(t)={\big(g_k(t)-g_{k+1}(t)\big)tg_k(t)+\big(2-tg_k(t)\big)g_{k+1}(t) \over 2}$$ 
implying $~[tg_{k+1}]'\geq 0~$ for all $t \geq 0$ as $g_k(t)-g_{k+1}(t) \geq 0$ and $2-tg_k(t)>0$. Hence, $tg_{k+1}(t)$ is bounded and nondecreasing. 
Thus, $~\lim\limits_{t \rightarrow \infty} tg_{k+1}(t)$ exists for all $k \geq 1$.

Next, suppose $~\lim\limits_{t \rightarrow \infty} tg_k(t) =2$. 
Then by the Mean Value Theorem, for any $t>0$ and for all $y>t$,
$${g_{k+1}(t)-g_{k+1}(y) \over t^{-1}-y^{-1}} \leq \sup\limits_{z:~z \geq t}{g'_{k+1}(z) \over -z^{-2}}.$$
Taking $~y \rightarrow \infty$, obtain
$${g_{k+1}(t) \over t^{-1}} \leq \sup\limits_{z:~z \geq t}{g'_{k+1}(z) \over -z^{-2}}.$$
Therefore
$$\lim\limits_{t \rightarrow \infty} tg_{k+1}(t)=\lim\limits_{t \rightarrow \infty} {g_{k+1}(t) \over t^{-1}}=\limsup\limits_{z \rightarrow \infty} {g'_{k+1}(z) \over -z^{-2}}=\limsup\limits_{z \rightarrow \infty} {{g^2_k(z) \over 2}-g_k(z) g_{k+1}(z) \over -z^{-2}}$$
$$=\limsup\limits_{z \rightarrow \infty} \left[z^2g_k(z) g_{k+1}(z)-{z^2 g^2_k(z) \over 2}\right]=2\lim\limits_{t \rightarrow \infty} tg_{k+1}(t) -2$$
implying $~\lim\limits_{t \rightarrow \infty} tg_{k+1}(t)=2$.
The statement (d) follows from the tree construction process. 
An alternative proof of (d) using differential equations is given in the following subsection. 
Part (e) follows from part (a) together with H\"older inequality
$${1 \over 2}\|g_k\|^2_{L^2[0, \infty)}
=\int\limits_0^{\infty} g_k(t)g_{k+1}(t) dt 
\leq \|g_k\|_{L^2[0, \infty)} \cdot \|g_{k+1}\|_{L^2[0, \infty)},$$
which implies ${\|g_{k}\|^2_{L^2[0, \infty)} \over \|g_{k+1}\|^2_{L^2[0, \infty)}} \leq {4}$.
\end{proof}

\bigskip
\noindent
Finally, observe that $g_k(t) \rightarrow 0$ as $k \rightarrow \infty$. 
Indeed, Proposition \ref{prop1} and the Dominated Convergence Theorem imply
$$\int\limits_0^{\infty} g^2_{k+1}(t) dt=\int\limits_0^{\infty} (g_k(t)-g_{k+1}(t))^2 dt \rightarrow 0 ~~\text{ as }k \rightarrow \infty.$$
Next, following (\ref{odeG}),
$$g_{k+1}(t)=\int\limits_0^t g'_{k+1}(y)dy=\int\limits_0^t{g^2_k(y) \over 2} dy - 
\int\limits_0^t g_k(y) g_{k+1}(y)dy \rightarrow 0 ~~\text{ as }k \rightarrow \infty.$$

\subsection{Rescaling to $[0,1]$ interval}

Let
$$h_k(x)=(1-x)^{-1}-(1-x)^{-2} g_{k+1}\left({2x \over 1-x}\right)$$
for $x \in [0,1]$.
Then $h_0 \equiv 0$, $h_1 \equiv 1$, and the system of ODEs (\ref{odeG}) rewrites as
\begin{equation} \label{ODEh}
h'_{k+1}(x)=2h_k(x)h_{k+1}(x)-h_k^2(x)
\end{equation}
with the initial conditions $h_k(0)=1$.

\medskip 
\noindent
Observe that the above quasilinearized system of ODEs (\ref{ODEh}) has $h_k(x)$ converging to $h(x)={1 \over 1-x}$ as $k \rightarrow \infty$, where $h(x)$ is the solution to Riccati equation $h'(x)=h^2(x)$ over $[0,1)$, with the initial value $h(0)=1$. Specifically, we have proven that $g_k(x) \rightarrow 0$ as $k \rightarrow \infty$. Thus
$$h_k(x)=(1-x)^{-1}-(1-x)^{-2} g_{k+1}\left({2x \over 1-x}\right) ~\rightarrow ~h(x)={1 \over 1-x}.$$ 
Here the quantity $n_k$ rewrites in terms of $h_k$ as follows
$$n_k={\big\|1-h_{k+1}/h\big\|^2_{L^2[0,1]} \over \big\|1-h_k/h\big\|^2_{L^2[0,1]} }.$$

\noindent
Observe that $h_2(x)=(1+e^{2x})/2$, but for $k \geq 3$ finding a closed form expression becomes increasingly hard. 
Given $h_k(x)$, Eq.~\eqref{ODEh} is a linear first-order ODE in $h_{k+1}(x)$;
its solution is given by $h_{k+1}(x)=\mathcal{H}h_k(x)$ with
\be
\label{iter}
\mathcal{H}f(x)=\left[1-\int_0^x f^2(y) e^{-2\int\limits_0^y f(s) ds } dy \right] 
\cdot e^{2\int\limits_0^x f(s) ds }.
\ee
Hence, the problem we are dealing with concerns the asymptotic behavior of an iterated
non-linear functional.

Using the setting of (\ref{ODEh}), we give an ODE proof to Proposition \ref{prop1}(d). 
To do so, we first need to prove the following lemma.

\begin{lem}\label{half}
$$\big\|1-h_{k+1}/h\big\|_{L^2[0,1]}=\big\|h_{k+1}/h-h_k/h\big\|_{L^2[0,1]}$$
\end{lem}
\begin{proof} 
Observe that $$h'_{k+1}(x)+(h_{k+1}(x)-h_k(x))^2=h_{k+1}^2(x).$$
We now use integration by parts to obtain
$$\int\limits_0^1 {(h_{k+1}(x)-h_k(x))^2 \over h^2(x)}dx=\int\limits_0^1 {h_{k+1}^2(x) \over h^2(x)}dx-\int\limits_0^1{h'_{k+1}(x) \over h^2(x)}dx$$
$$=\int\limits_0^1 {h_{k+1}^2(x) \over h^2(x)}dx+1-2\int\limits_0^1 {h_{k+1}(x) \over h(x)}dx$$
$$=\int\limits_0^1 {(1-h_{k+1}(x))^2 \over h^2(x)}dx$$
since $1/h(x)=1-x$.
\end{proof}
\vskip 0.2 in
\noindent
\begin{proof} [Alternative proof of Proposition \ref{prop1}(d)]
Notice that 
$h \geq \dots\geq h_{k+1} \geq h_k\geq \dots \geq h_0 \equiv 0$, 
which follows from 
$g_1(t) \geq g_2(t) \geq g_3(t) \geq \dots$
The Lemma \ref{half} implies
$$\big\|1-h_{k+1}/h\big\|_{L^2[0,1]}^2 = \big\|h_{k+1}/h-h_k/h\big\|_{L^2[0,1]}^2= \int\limits_0^1 \left[(1-h_k/h)-(1-h_{k+1}/h) \right]^2dx$$
$$= \big\|1-h_{k+1}/h\big\|_{L^2[0,1]}^2
+\big\|1-h_k/h\big\|_{L^2[0,1]}^2 -2\int\limits_0^1(1-h_k/h)(1-h_{k+1}/h) dx $$
and therefore
$$\big\|1-h_k/h\big\|_{L^2[0,1]}^2 =2\int\limits_0^1(1-h_k/h)(1-h_{k+1}/h) dx$$
$$=2\big\|1-h_{k+1}/h\big\|_{L^2[0,1]}^2 
+2\int\limits_0^1(h_{k+1}/h-h_k/h)(1-h_{k+1}/h) dx \geq 2\big\|1-h_{k+1}/h\big\|_{L^2[0,1]}^2$$
yielding $2 \leq {\big\|1-h_{k}/h\big\|_{L^2[0,1]}^2 \over \big\|1-h_{k+1}/h\big\|_{L^2[0,1]}^2}=n_{k}~$ as in Proposition \ref{prop1}(d). 
\end{proof}
It is also true that one can improve Proposition \ref{prop1}(d) to make it a 
strict inequality since one can check that 
$$h(x) > \dots > h_{k+1}(x) > h_k(x)> \dots > h_0(x) \equiv 0 \quad \text{ for } x \in (0,1).$$

\subsection{Proof of the existence of the root-Horton limit}

Here we present the proof of our main Theorem \ref{main}.
It is based on Lemma \ref{h1} and Lemma \ref{h1exist} 
that will be proven in the following two subsections. 

\begin{lem} \label{h1}
If the limit $\lim\limits_{k \rightarrow \infty}{h_{k+1}(1) \over h_k(1)}$ exists, then 
$\lim\limits_{k \rightarrow \infty}\left(\cN_k \right)^{-{1 \over k}}=\lim\limits_{k \rightarrow \infty}\left(\prod\limits_{j=1}^{k} n_j \right)^{-{1 \over k}}$ also exists, and
$$\lim\limits_{k \rightarrow \infty}\left(\cN_k \right)^{-{1 \over k}}
=\lim\limits_{k \rightarrow +\infty}\left({1 \over h_k(1)}\right)^{-{1 \over k}}
=\lim\limits_{k \rightarrow \infty}{h_{k+1}(1) \over h_{k}(1)}.$$
\end{lem}

\begin{lem} \label{h1exist}
The limit $\lim\limits_{k \rightarrow \infty}{h_{k+1}(1) \over h_k(1)} \geq 1$ exists, and is finite.
\end{lem}
 
\begin{thm} \label{main}
The limit 
$\lim\limits_{k \rightarrow \infty}\left(\cN_k \right)^{-{1 \over k}}
=\lim\limits_{k \rightarrow \infty}\left(\prod\limits_{j=1}^{k} n_j \right)^{-{1 \over k}}=R$ exists.
Moreover, 
$R=\lim\limits_{k \rightarrow \infty} {h_{k+1}(1) \over h_{k}(1)}$, and  
$2 \leq R \leq 4$.
\end{thm}

\noindent
\begin{proof}
The existence and finiteness of $\lim\limits_{k \rightarrow \infty}{h_{k+1}(1) \over h_k(1)}$ 
established in Lemma \ref{h1exist} is the precondition for Lemma \ref{h1} that in turn implies 
the existence and finiteness of the limit 
$\lim\limits_{k \rightarrow \infty}\left(\cN_k \right)^{-{1 \over k}}$ 
as needed for the root-Horton law.
Finally, $2 \leq R \leq 4$ follows from Proposition \ref{prop1}.
\end{proof}

\subsection{Proof of Lemma \ref{h1} and related results}

\begin{prop}\label{one}
$$\big\|1-h_{k+1}(x)/h(x)\big\|_{L^2[0,1]}^2 
\leq {1 \over h_{k+1}(1)} \leq \big\|1-h_k(x)/h(x)\big\|_{L^2[0,1]}^2.$$
\end{prop}
\begin{proof}
Integrating from 0 to 1 both sides of the equation
$${h'_{k+1}(x) \over h_{k+1}^2(x)}=1-{(h_{k+1}(x)-h_k(x))^2 \over h_{k+1}^2(x)}$$
we obtain $~{1 \over h_{k+1}(1)}=\int\limits_0^1 {(h_{k+1}(x)-h_k(x))^2 \over h_{k+1}^2(x)} dx~$ as $h_{k+1}(0)=1$.

\vskip 0.2 in
\noindent
Hence,
$${1 \over h_{k+1}(1)}=\int\limits_0^1 {(h_{k+1}(x)-h_k(x))^2 \over h_{k+1}^2(x)} dx \geq \int\limits_0^1 {(h_{k+1}(x)-h_k(x))^2 \over h^2(x)} dx =\int\limits_0^1 \left( 1-{h_{k+1}(x) \over h(x)}\right)^2 dx$$
by Lemma \ref{half}, proving the first inequality.

\vskip 0.2 in
\noindent
Now,
$${1 \over h_{k+1}(1)}=\big\|1-h_k(x)/h_{k+1}(x)\big\|_{L^2[0,1]}^2 \leq \big\|1-h_k(x)/h(x)\big\|_{L^2[0,1]}^2$$
thus completing the proof.
\end{proof}


\begin{proof}[Proof of Lemma \ref{h1}]
If the limit $\lim\limits_{k \rightarrow \infty}{h_{k+1}(1) \over h_k(1)}$ exists and is finite, then $\lim\limits_{k \rightarrow \infty}\left({1 \over h_k(1)}\right)^{-{1 \over k}}$ must also exist and be finite. Hence the existence and finiteness of 
$$\lim\limits_{k \rightarrow \infty}\left(\cN_k \right)^{-{1 \over k}}=\lim\limits_{k \rightarrow \infty}\left(\int_0^1 \left(1-{h_k(x)\over h(x)}\right)^2 dx \right)^{-{1 \over k}}$$
follows from Proposition \ref{one}.
\end{proof}

\subsection{Proof of Lemma \ref{h1exist} and related results}
In this subsection we use the approach developed by Drmota \cite{MD2009} 
to prove the existence and finiteness of 
$\lim\limits_{k \rightarrow \infty}{h_{k+1}(1) \over h_k(1)} \geq 1$. 
As we observed earlier this result is needed to prove the existence, finiteness, and positivity  of  
$\lim\limits_{k \rightarrow \infty}\left(\cN_k \right)^{-{1 \over k}}=\lim\limits_{k \rightarrow \infty}\left(\prod\limits_{j=1}^{k} n_j \right)^{-{1 \over k}} $, the root-Horton law.

\begin{Def}
Given $\gamma \in  (0,1]$. Let
$$V_{k,\gamma}(x)=\begin{cases}
      {1 \over 1-x} & \text{ for } 0 \leq x \leq 1-\gamma, \\
      \gamma^{-1} h_k\left({x-(1-\gamma) \over \gamma}\right) & \text{ for } 1-\gamma \leq x \leq 1.
\end{cases} $$
\end{Def}
\noindent
Note that sequences of functions $h_k(x)$ and $V_{k,\gamma}(x)$ can be extended beyond $x=1$.

\noindent
Here are some observations we make about the above defined functions. 

\begin{obs}
$V_{k,\gamma}(x)$ are positive continuous functions satisfying
$$V'_{k+1,\gamma}(x)=2V_{k+1,\gamma}(x)V_{k,\gamma}(x)-V^2_{k,\gamma}(x)$$
 for all $x \in [0,1] \setminus (1-\gamma)$, with initial conditions $V_{k,\gamma}(0)=1$.
\end{obs}

\begin{obs} 
Let $\gamma_k={h_k(1) \over h_{k+1}(1)}$.  Then
\begin{equation} \label{gammak1}
V_{k,\gamma_k}(1)=h_{k+1}(1)
\end{equation}
and
\begin{equation} \label{gammak2}
V_{k,\gamma}(1)=\gamma^{-1} h_k(1) \geq h_{k+1}(1) \quad \text{ whenever } \gamma \leq \gamma_k.
\end{equation}
\end{obs}

\begin{obs}
$$V_{k,\gamma}(x) \leq V_{k+1,\gamma}(x)$$
for all $x \in [0,1]$ since $h_k(x) \leq h_{k+1}(x)$.
\end{obs}

\begin{obs}  
Since $h_1(x) \equiv 1$ and $\gamma_1={h_1(1) \over h_2(1)}$, 
$$h_2(x) \leq V_{1,\gamma_1}(x)=\begin{cases}
      {1 \over 1-x} & \text{ for } 0 \leq x \leq 1-\gamma_1, \\
      \gamma_1^{-1}=h_2(1) & \text{ for } 1-\gamma_1 \leq x \leq 1.
\end{cases} $$
\end{obs}

The above observation generalizes as follows.
\begin{prop}\label{drmota}
$$h_{k+1}(x) \leq V_{k,\gamma_k}(x)=\begin{cases}
      {1 \over 1-x} & {\rm~for~} 0 \leq x \leq 1-\gamma_k, \\
     \gamma_k^{-1} h_k\left({x-(1-\gamma_k) \over \gamma_k}\right) 
     & {\rm~ for ~} 1-\gamma_k \leq x \leq 1.
\end{cases}$$
\end{prop}
In order to prove Proposition \ref{drmota} we will need the following lemma.
\begin{lem} 
\label{positivezero}
For any $\gamma \in (0,1)$ and $k \geq 1$, function $V_{k,\gamma}(x)-h_{k+1}(x)$ changes its sign at most once as $x$ increases from $1- \gamma$ to $1$. Moreover, since $V_{k,\gamma}(1-\gamma)=h(1-\gamma) > h_{k+1}(1-\gamma)$,  function $V_{k,\gamma}(x)-h_{k+1}(x)$ can only change sign from nonnegative to negative.
\end{lem}
\begin{proof}
This is a proof by induction with base at $k=1$. Here $V_{1,\gamma}(x) ={1 \over \gamma}$ is constant on $[1-\gamma,1]$, while $h_2(x)=(1+e^{2x})/2$ is an increasing function, and 
$$V_{1,\gamma}(1-\gamma)=h(1-\gamma)>h_2(1-\gamma)$$

For the induction step, we need to show that if $V_{k,\gamma}(x)-h_{k+1}(x)$ changes its sign at most once, 
then so does $V_{k+1,\gamma}(x)-h_{k+2}(x)$. 
Since both sequences of functions satisfy the same ODE relation (see Observation 1), we have

\noindent
${d \over dx}\left[(V_{k+1,\gamma}(x)-h_{k+2}(x))\cdot e^{-2\int\limits_{1-\gamma}^x h_{k+1}(y)dy} \right]$
$$\qquad \qquad =(2V_{k+1,\gamma}(x)-V_{k,\gamma}(x)-h_{k+1}(x))\cdot (V_{k,\gamma}(x)-h_{k+1}(x))\cdot e^{-2\int\limits_{1-\gamma}^x h_{k+1}(y)dy},$$
where $h_{k+1}(x) \leq V_{k+1,\gamma}(x)$ by definition of $V_{k+1,\gamma}(x)$, and $V_{k,\gamma}(x) \leq V_{k+1,\gamma}(x)$ as in Observation 3. 
\vskip 0.2 in
\noindent
Now, let
$$I(x):=\int\limits_{1-\gamma}^x (2V_{k+1,\gamma}(s)-V_{k,\gamma}(s)-h_{k+1}(s))\cdot (V_{k,\gamma}(s)-h_{k+1}(s))\cdot e^{-2\int\limits_{1-\gamma}^s h_{k+1}(y)dy} ds.$$
Then 
$$~(V_{k+1,\gamma}(x)-h_{k+2}(x))\cdot e^{-2\int\limits_{1-\gamma}^x h_{k+1}(y)dy}=V_{k+1,\gamma}(1-\gamma)-h_{k+2}(1-\gamma)+I(x).$$
\vskip 0.2 in
\noindent
The function $2V_{k+1,\gamma}(x)-V_{k,\gamma}(x)-h_{k+1}(x) \geq 0$, and since 
$V_{k,\gamma}(x)-h_{k+1}(x)$ changes its sign at most once, then $I(x)$ should change its sign from nonnegative to negative at most once as $x$ increases from $1-\gamma$ to $1$. Hence
$$V_{k+1,\gamma}(x)-h_{k+2}(x)=(V_{k+1,\gamma}(1-\gamma)-h_{k+2}(1-\gamma)+I(x)) \cdot e^{2\int\limits_{1-\gamma}^x h_{k+1}(y)dy}$$
should change its sign from nonnegative to negative at most once as 
$$V_{k+1,\gamma}(1-\gamma)=h(1-\gamma)>h_{k+2}(1-\gamma).$$
\end{proof}

\begin{proof}[Proof of Proposition \ref{drmota}]
Take $\gamma=\gamma_k$ in Lemma \ref{positivezero}. Then function $h_{k+1}(x)-V_{k,\gamma_k}(x)$ 
should change its sign from nonnegative to negative at most once within the interval $[1-\gamma_k,1]$. 
Hence, $V_{k,\gamma_k}(1-\gamma_k) > h_{k+1}(1-\gamma_k)$ and $h_{k+1}(1) = V_{k,\gamma_k}(1)$ imply 
$h_{k+1}(x) \leq V_{k,\gamma_k}(x)$ as in the statement of the proposition.
\end{proof}


Now we are ready to prove the monotonicity result.
\begin{lem} \label{gamma}
$$\gamma_k \leq \gamma_{k+1} \qquad \text{ for all } k \in \mathbb{N}^+.$$
\end{lem}
\begin{proof} We prove it by contradiction. 
Suppose $\gamma_k \geq \gamma_{k+1}$ for some $k \in \mathbb{N}^+$. 
Then
$$V_{k,\gamma_k}(x) \leq V_{k,\gamma_{k+1}}(x)=\begin{cases}
      {1 \over 1-x} & \text{ for } 0 \leq x \leq 1-\gamma_{k+1}, \\
     \gamma_{k+1}^{-1} h_k\left({x-(1-\gamma_{k+1}) \over \gamma_{k+1}}\right) 
     & \text{ for } 1-\gamma_{k+1} \leq x \leq 1
\end{cases}$$
and therefore
$$h_{k+1}(x) \leq V_{k,\gamma_k}(x) \leq V_{k,\gamma_{k+1}}(x) \leq V_{k+1,\gamma_{k+1}}(x)$$
as $h_{k+1}(x) \leq V_{k,\gamma_k}(x)$ by Proposition \ref{drmota}.

Recall that for $x \in [1-\gamma_{k+1},1]$, 
$$V'_{k+1,\gamma_{k+1}}(x)=2V_{k,\gamma_{k+1}}(x)V_{k+1,\gamma_{k+1}}(x)-V_{k,\gamma_{k+1}}^2,$$ 
where at $1-\gamma_{k+1}$ we consider only the right-hand derivative.
Thus for $x \in [1-\gamma_{k+1},1]$,
$${d \over dx}\Big(V_{k+1,\gamma_{k+1}}(x)-h_{k+2}(x)\Big)=A(x)+B(x)\Big(V_{k+1,\gamma_{k+1}}(x)-h_{k+2}(x)\Big),$$
where $A(x)=2V_{k+1,\gamma_{k+1}}(x)-V_{k,\gamma_{k+1}}(x)-h_{k+1}(x) \geq 0$, $B(x)=2h_{k+1}(x) >0$, and $V_{k+1,\gamma_{k+1}}(1-\gamma_{k+1})-h_{k+2}(1-\gamma_{k+1})=h(1-\gamma_{k+1})-h_{k+2}(1-\gamma_{k+1})>0$.
Hence $$V_{k+1,\gamma_{k+1}}(1) - h_{k+2}(1) \geq V_{k+1,\gamma_{k+1}}(1-\gamma_{k+1})-h_{k+2}(1-\gamma_{k+1})>0$$
arriving to a contradiction since $V_{k+1,\gamma_{k+1}}(1) = h_{k+2}(1)$.
\end{proof}

\vskip 0.2 in
\noindent
\begin{cor*}
Limit $\lim\limits_{k \rightarrow \infty} \gamma_k$ exists.
\end{cor*}
\begin{proof}
Lemma \ref{gamma} implies  $\gamma_k$ is a monotone increasing sequence, bounded by $1$.
\end{proof}

\vskip 0.1 in
\noindent
\begin{proof}[Proof of Lemma \ref{h1exist}]
Lemma \ref{h1exist} follows immediately from an observation 
that ${h_{k+1}(1) \over h_k(1)}={1 \over \gamma_k}$.
\end{proof}

\section{Relation to the tree representation of white noise} 
\label{white}
This section establishes a close connection between the combinatorial tree
of Kingman's $N$-coalescent 
and the combinatorial level set tree of a discrete white noise. 

\subsection{Level set tree of a discrete-time function}
\label{level}
We start with recalling basic facts about tree representation of a 
discrete-time function; for details and further results see \cite{ZK12}.
Consider a function $X_i$ with discrete time index $i=0,1,\dots,i_{\rm max}$
and values distributed without atoms over $\mathbb{R}$.
Let $X_t\equiv X(t)$ be a function of continuous time $t\in[0,i_{\rm max}]$ obtained from
$X_i$ by linear interpolation of its values.
The level set $\mathcal{L}_{\alpha}\left(X_t\right)$ is defined 
as the pre-image of the function values above $\alpha$: 
\[\mathcal{L}_{\alpha}\left(X_t\right) = \{t\,:\,X_t\ge\alpha\}.\]
The level set $\mathcal{L}_{\alpha}$ for each $\alpha$ is
a union of non-overlapping intervals; we write 
$|\mathcal{L}_{\alpha}|$ for their number.
Notice that 
$|\mathcal{L}_{\alpha}| = |\mathcal{L}_{\beta}|$ 
as soon as the interval $[\alpha,\,\beta]$ does not contain a value of
local maxima or minima of $X_t$ and  
$0\le |\mathcal{L}_{\alpha}| \le n$, where $n$ is the number 
of the local maxima of $X_t$. 

The {\it level set tree} $\textsc{level}(X_t)$ is a planar time oriented binary
tree that describes the topology of the level sets $\mathcal{L}_{\alpha}$ 
as a function of threshold $\alpha$, as illustrated in Fig.~\ref{fig3}.
Namely, there are bijections between 
(i) the leaves of $\textsc{level}(X_t)$ and the local
maxima of $X_t$,
(ii) the internal (parental) vertices of $\textsc{level}(X_t)$ 
and the local minima of $X_t$ (excluding possible local minima
at the boundary points), and
(iii) the pair of subtrees of $\textsc{level}(X_t)$ rooted at a local
minima $X(t^*)$ and the first positive excursions (or meanders bounded
by $t=0$ or $t=N$) of $X(t)-X(t^*)$ to right and left of $t^*$.
Each vertex in the tree is assigned a mark equal to the value
of the local extrema according to the bijections (i) and (ii) above.
This makes the tree time oriented according to the threshold $\alpha$.
It is readily seen that any function $X_t$ with distinct 
values of consecutive local minima corresponds to a binary tree 
$\textsc{level}(X_t)$.
We refer to \cite{ZK12} for discussion of some subtleties 
related to this construction as well as for further references. 

\subsection{Tree representation of white noise}
\label{finite}
Let $W^{(N)}_j$, $j=1,\dots,N-1$, be a {\it discrete white noise} that is 
a discrete time process comprised of $N-1$ i.i.d. random variables 
with a common atomless distribution.
Consider now an auxiliary process $\tilde W^{(N)}_{i}$, $i=1,\dots,2N-1$ such that
it has exactly $N$ local maxima and $N-1$ internal local minima 
$\tilde W^{(N)}_{2j}=W^{(N)}_j$, $j=1,\dots,N-1$.
We call $\tilde W^{(N)}_{i}$ an {\it extended white noise};
it can be constructed, for example, as follows: 
\be
\label{wnt}
\tilde W^{(N)}_{i}=
\left\{
\begin{array}{cc}
W^{(N)}_{i/2},& {\rm for~even~}i,\\
\max\left(W^{(N)}_{\max\left(1,\frac{i-1}{2}\right)},W^{(N)}_{\min\left(N-1,\frac{i+1}{2}\right)}\right)+1,&{\rm for~odd~}i.
\end{array}
\right.
\ee

Let $L^{(N)}_W=\textsc{level}\left(\tilde W^{(N)}_i\right)$ be the level set tree 
of $\tilde W^{(N)}_i$ and 
$\textsc{shape}\left(L^{(N)}_W\right)$ be a (random) combinatorial tree 
that retains the graph-theoretic structure of
$L^{(N)}_W$ and drops its planar embedding as well as the 
vertex marks.
By construction, $L^{(N)}_W$ has exactly $N$ leaves.

\begin{lem}
\label{any_wh}
The distribution of $\textsc{shape}\left(L^{(N)}_W\right)$ on $\cT_N$
is the same for any atomless distribution $F$ of the values of the
associated white noise $W^{(N)}_j$.
\end{lem}
\begin{proof}
The condition of atomlessness of $F$ is necessary to ensure that the
level set tree is binary with probability 1.
By construction, the combinatorial level set tree is completely determined by the 
ordering of the local minima of the respective trajectory, independently of
the particular values of its local maxima and minima.
We complete the proof by noticing that the ordering of $W^{(N)}_j$ is 
the same for any choice of atomless distribution $F$.  
\end{proof}

Let $T^{(N)}_{\rm K}$ be the tree that corresponds to a 
Kingman's $N$-coalescent, and let
$\textsc{shape}\left(T^{(N)}_{\rm K}\right)$ be its combinatorial version that
drops the time marks of the vertices.
Both the trees $\textsc{shape}\left(L^{(N)}_W\right)$ and 
$\textsc{shape}\left(T^{(N)}_{\rm K}\right)$, belong to the space $\cT_N$ of 
binary rooted trees with $N$ leaves.

\begin{thm}\label{main2}
The trees $\textsc{shape}\left(L^{(N)}_W\right)$ and 
$\textsc{shape}\left(T^{(N)}_{\rm K}\right)$ have the same distribution on $\cT_N$.
\end{thm}  

The proof below uses the duality between coalescence and fragmentation
processes \cite{Aldous}.
Recall that a {\it fragmentation process} starts with a single cluster of
mass $N$ at time $t=0$.
Each existing cluster of mass $m$ splits into two clusters
of masses $m-x$ and $x$ at the splitting rate $S_t(m,x)$, $1<m\le N$, $1\le x < N$. 
A coalescence process on $N$ particles with time-dependent 
collision kernel $K_t(x,y)$, $1\le x,y < N$ is equivalent, upon time 
reversal, to a discrete-mass fragmentation process of initial mass 
$N$ with some splitting kernel $S_t(m,x)$.
See Aldous \cite{Aldous} for further details and the relationship
between the dual collision and splitting kernels in general case.

\begin{proof}[Proof of Theorem~\ref{main2}]
We show that both the examined trees have the same distribution as
the combinatorial tree of a fragmentation process with mass $N$ and a 
splitting kernel that is uniform in mass:
$S_t(m,x)=S(t).$

Kingman's $N$-coalescence with kernel $K(x,y)= 1$ 
is dual to the fragmentation process with splitting kernel \cite[Table 3]{Aldous}
\[S_t(m,x) = \frac{2}{t\,(t+2)}.\]
This kernel is independent of the cluster mass, which means
that the splitting of mass $m$ is uniform among the $m-1$ possible pairs 
$\{1,m-1\}$, \mbox{$~\{2,m-2\}$}, $\hdots,\{m-1,1\}$.
The time dependence of the kernel does not affect the combinatorial structure
of the fragmentation tree (and can be removed by a deterministic time change.) 

The level set tree $L^{(N)}_W$ can be viewed as a tree that describes 
a fragmentation process with the initial mass $N$ equal
to the number of local maxima of the trajectory $\tilde W^{(N)}_i$. 
By construction, each subtree of $L^{(N)}_W$ with $n$ 
leaves corresponds to an excursion (or meander, if we treat one of the 
boundaries) with $n$ local maxima. 
This subtree (as well as the corresponding excursion or meander) splits 
into two by the internal global minimum of $\tilde W^{(N)}_i$ at 
the corresponding time interval. 

The global minimum splits the series $\tilde W^{(N)}_i$ into two,
to the left and right of the minimum, with $M_L$ and $(N-M_L)$ local maxima,
respectively.
Since the local minima of $\tilde W^{(N)}_i$ form a white noise, the 
distribution of $M_L$ is uniform on $[1,N-1]$.
Next, the internal vertices of the level set tree of the left (or right) time 
series correspond to its $M_L-1$ (or $N-M_L-1$) internal local minima that form 
a white noise (with the distribution different from that of the initial
white noise $W^{(N)}_j$). 
Hence, the subsequent splits of masses (number of local maxima) continues
according to a discrete uniform distribution.
And so on down the tree.

Hence, the combinatorial level set tree of $\tilde W^{(N)}_i$ 
has the same distribution as a combinatorial tree of a fragmentation 
process with uniform mass splitting.
This completes the proof.
\end{proof}

\begin{Rem}
We notice that the dual splitting kernels for multiplicative and additive
coalescences \cite[Table 3]{Aldous} only differ by their time dependence,
and are equivalent as functions of mass. 
Hence, the combinatorial structure of the respective trees is the same.
\end{Rem}

\begin{cor}
\label{main3}
The combinatorial level set tree of a discrete white noise $W^{(N)}$ 
is root-Horton self similar with the same Horton exponent $R$ as 
that for Kingman's $N$-coalescent.
\end{cor}

\begin{proof}
Recall the operation of tree {\it pruning}
$\cR(T):\cT\to\cT$ that cuts the leaves of a finite tree $T$ and removes 
possible resulting nodes of degree 2 \cite{BWW00,ZK12}.
By definition, pruning corresponds to index shift in Horton statistics:
$N_k\to N_{k-1}$, $k>1$. 
It has been shown in \cite{ZK12} that 
\[\cR\left[\textsc{level}\left(\tilde W^{(N)}_i\right) \right]=
\textsc{level}\left(W^{(N)}_j\right).\]
Hence, Horton self-similarity for one of these processes implies that
for the other. 
The Horton self-similarity for the extended white noise $\tilde W^{(N)}$
follows directly from Theorem~\ref{main2}.
\end{proof}

\section{General coalescent processes} 
\label{general}
The ODE approach introduced in this paper can be extended to the 
coalescent kernels other than $K(i,j) \equiv 1$. 
For that we need to classify the relative number $\eta_k(t)$ of clusters of 
order $k$ at time $t$ according to the cluster masses. 
Namely, let $\eta_{k,m}(t)$ be the average number of clusters of order $k$ 
and mass $m \geq 2^k$ at time $t$.
Then
$$\eta_k(t)=\sum\limits_{m=2^k}^{\infty} \eta_{k,m}(t).$$

In the case of a symmetric coalescent 
kernel $K(i,j)=K(j,i)$ the Smoluchowski-Horton ODEs can be written asymptotically as 
\begin{eqnarray}\label{ODEker}
{d \over dt}\eta_{k,m}(t) & = &\sum\limits_{i=1}^{k-1} \sum\limits_{\mu=2^k}^{m-2^i} 
\eta_{k,\mu}(t)\eta_{i,m-\mu} K(\mu, m-\mu)\\ \nonumber
& + & {1 \over 2} \sum_{\substack{m_1+m_2=m \\ m_1,m_2 \geq 2^{k-1}}} 
\eta_{k-1,m_1}(t) \eta_{k-1,m_2}(t) K(m_1,m_2)\\ \nonumber
& - & \eta_{k,m}(t)\sum\limits_{\widetilde{m}=2^i}^{\infty} K(m,\widetilde{m})  
\left(\sum\limits_{i=1}^{\infty} \eta_{i, \widetilde{m}}(t)\right)
\end{eqnarray} 
with the initial conditions $\eta_{1,1}(0)=1~$ and $~\eta_{k,m}(0)=0$ for all $(k,m) \not=(1,1)$.

Observe that when $K(i,j) \equiv 1$, summing the above equations (\ref{ODEker}) over index 
$m$ produces the  Smoluchowski-Horton ODE (\ref{Aeta}) for the average relative number of 
order-$k$ branches $\eta_k(t)$ in Kingman's coalescent process.

\section{Discussion}
\label{discussion} 
This paper establishes the root-Horton self-similarity (Sect.~\ref{existence}, Thm~\ref{main}) for 
Kingman's $N$-coalescent process, as $N$ goes to infinity.
We also demonstrate (Sect.~\ref{level}, Thm~\ref{main2}) the distributional equivalence of the 
combinatorial trees of Kingman's $N$-coalescent to that of a discrete 
extended white noise with $N$ local maxima, hence extending the 
self-similarity results to a tree representation of a discrete white noise
(Sect.~\ref{white}, Cor~\ref{main3}).

Combining the results of this study with that of Burd et al. \cite{BWW00} and 
Zaliapin and Kovchegov \cite{ZK12} one observes that Horton self-similarity 
is a property of 
(i) white noise, 
(ii) symmetric random walk, 
(iii) critical binary Galton-Watson branching process, and
(iv) Kingman's $N$-coalescent.
The listed processes are believed to closely depict physical and
biological mechanisms of diverse origin and are commonly used as 
essential building blocks in scientific modeling.
The results of this study and those in \cite{BWW00,ZK12} thus provide at 
least a partial explanation for the omnipresence of Horton self-similarity in 
observed and modeled branching structures.
This study seems to be the first that rigorously establishes
Horton self-similarity with Horton exponent different from $R=2,4$.

Our Theorem~\ref{main} establishes a weak, root-law, convergence of the
asymptotic ratios $\cN_k$, 
while we believe that the stronger (ratio and geometric) forms of convergence
are also valid.
These stronger Horton laws are usually considered in
the literature (e.g., \cite{Pec95,Horton45,DR00,ZK12}). 
It seems important to show rigorously at least the ratio-Horton law
($\lim\limits_{k \rightarrow \infty} 
\cN_k/\cN_{k+1}=R>0$). 

The Smoluchowski-Horton equations \eqref{Aeta} that form a core of the 
presented method and their equivalents \eqref{odeG} and \eqref{ODEh}
seem to be promising for further more detailed exploration.
Indeed, one may hope that the approach that refers explicitly to the 
Horton-Strahler orders might effectively complement conventional analysis 
of cluster masses.
The analysis of the Smoluchowski-Horton systems can be done within
the ODE framework, similarly to the present study, or within the
nonlinear iterative system framework (see \eqref{iter}).
The latter approach is still to be explored.

Finally, it is noteworthy that the analysis of multiplicative and additive 
coalescents according to the general Smoluchowski-Horton 
system~\eqref{ODEker} appears, after a certain series of transformations, 
to follow many of the steps implemented in this paper for 
Kingman's coalescent, with the ODE system being replaced by a 
suitable PDE one.
These results will be published elsewhere. 

\vspace{1cm}
{\bf Acknowledgement.}
We are grateful to Ed Waymire for encouragement and continuing interest to this work.
We thank the participants of the 2012 Oregon State University
Workshop on Mathematical Problems in the Environmental Sciences for their constructive
feedback. 
Suggestions of an anonymous reviewer and the Associate Editor helped 
to significantly improve the original manuscript.

\appendix
\section{Proof of Lemma \ref{lem2}}
\label{AA2}

\begin{proof}  
We split the proof into smaller steps.

\bigskip
\noindent
$\bullet$ \textit{Step I.} Fix $\epsilon_0 \in (0,1)$ and take $\delta>0$. 
We show below that, given $\eta_{(N)}(t)= y \in {1 \over N} \mathbb{Z} \cap [\epsilon_0,1]$, 
the number of coalescences during the time interval $[t,t+\delta]$ 
does not exceed ${\delta \over N} \binom{N\,y}{2}+N^{2/3}$ with high probability. 
Specifically, we use exponential Markov inequality (aka Chernoff's bound) with exponent $s>0$ to bound the probability that a sum of 
${\delta \over N} \binom{N\,y}{2}+N^{2/3}$ exponential inter-arrival 
times with the rate not exceeding ${1  \over N} \binom{N\,y}{2}$ adds up to less than $\delta$.
Let $\zeta_i$ be the arrival time of $i$-th coalescence and $u=\binom{N\,y}{2}$. Then

\begin{eqnarray*}
& & \!\!\!\! \!\!\!\! \!\!\!\! \!\!\!\! \!\!\!\! \!\!\!\! \!\!\!\! \!\!\!\! \!\!\!\! \!\!\!\! \!\!\!\! 
P\left( N\big[\eta_{(N)}(t)-\eta_{(N)}(t+\delta) \big] >{\delta \over N} u+N^{2/3} ~~\Big| ~\eta_{(N)}(t)=y~\right)\\
&=& P\left(\sum_{i=1}^{\lfloor{\delta \over N} u+N^{2/3}\rfloor} \zeta_i <\delta~~\Big|
~\eta_{(N)}(t)=y~\right)\\
& \leq &  {e^{s \delta} \over \left(1+{s\,N\over u}\right)^{{\delta \over N} u+N^{2/3}}} \\
& \leq &  \exp\left\{s \delta - \Big({\delta \over N}  u+N^{2/3} \Big) \left({s\,N \over u}-{s^2\,N^2 \over u^2} \right) \right\} \\
&= & \exp\left\{-{s \over u}N^{5/3} +{\frac{\delta\,s^2}{u}\,N} +{s^2 \over u^2}N^{8/3} \right\} \\
\end{eqnarray*}
as $~\ln(1+x) > x-x^2~$ for $x>0$.
Taking $s=N^{1/2}$ in the above inequality, we obtain
\begin{eqnarray}\label{ineq:ubound}
& & \!\!\!\! \!\!\!\! \!\!\!\! \!\!\!\! \!\!\!\! \!\!\!\! \!\!\!\! \!\!\!\! \!\!\!\! \!\!\!\! \!\!\!\!  
P\left( N\big[\eta_{(N)}(t)-\eta_{(N)}(t+\delta) \big] >{\delta \over N} u+N^{2/3} ~~\Big| ~\eta_{(N)}(t)=y~\right) \nonumber \\
& = & \exp\left\{-{1 \over u}N^{13/6} +{\delta N^2 \over u} +{N^{11/3} \over u^2}\right\}  \nonumber \\
& = &  \exp\left\{-{2 \over Ny(Ny-1)}N^{13/6} +{2\delta N^2 \over Ny(Ny-1)} +{4N^{11/3} \over (Ny)^2(Ny-1)^2}\right\}  \nonumber \\
& = &  \exp\left\{-{2 \over y(y-1/N)}N^{1/6} +{2\delta \over y(y-1/N)} +{4N^{-1/3} \over (y)^2(y-1/N)^2}\right\}  \nonumber \\
& \leq & \exp\left\{-2N^{1/6} +{2\delta \over \epsilon_0(\epsilon_0-1/N)} +{4N^{-1/3} \over \epsilon_0^2(\epsilon_0-1/N)^2}\right\} \nonumber \\
& \leq & \exp\left\{-N^{1/6} +{4\delta \over \epsilon_0^2} \right\}
\end{eqnarray}

\noindent
for $N$ large enough.

\bigskip
$\bullet$ \textit{Step II.} From Step I we know that, given $\eta_{(N)}(t)=y \in {1 \over N} \mathbb{Z} \cap [\epsilon_0,1]$, there are no more than 
$${\delta \over N} \binom{N y}{2}+N^{2/3}={\delta y^2 \over 2}N-{\delta y \over 2}+N^{2/3} \leq {\delta y^2 \over 2}N+N^{2/3}$$ 
coalescing pairs during $[t,t+\delta]$ with probability exceeding $1-\exp\left\{-N^{1/6} +{4\delta \over \epsilon_0^2} \right\}$. 
In this case the exponential rates of inter-arrival times during $[t,t+\delta]$ must be at least 
$${1 \over N}\binom{N y-\lceil {\delta y^2 \over 2}N \rceil-\lceil N^{2/3}\rceil}{2} 
= {1 \over N}\binom{N y-\lceil {\delta y^2 \over 2}N \rceil}{2}-{N y-\lceil {\delta y^2 \over 2}N \rceil - 1/2 -{\lceil N^{2/3}\rceil/ 2} \over N} \lceil N^{2/3}\rceil$$
$$ \geq {1 \over N}\binom{N y-\lceil {\delta y^2 \over 2}N \rceil}{2}-\left(y- {\delta y^2 \over 2}\right)\lceil N^{2/3}\rceil \geq {1 \over N}\binom{N y-\lceil {\delta y^2 \over 2}N \rceil}{2}-N^{2/3}$$
for $N$ large enough. 
We now use exponential Markov inequality to bound the conditional probability that there are 
fewer than ${\delta \over N}\binom{N y-\lceil {\delta y^2 \over 2}N \rceil}{2}-(1+\delta)N^{2/3}$ coalescents in $[t,t+\delta]$. 
Specifically, we bound the probability that a sum of  
${\delta \over N}\binom{N y-\lceil {\delta y^2 \over 2}N \rceil}{2}-(1+\delta)N^{2/3}$ independent exponential random variables of rate not less than 
${1 \over N}\binom{N y-\lceil {\delta y^2 \over 2}N \rceil}{2}-N^{2/3}$ is greater than $\delta$. 

\medskip
\noindent
Set $v=\binom{N y-\lceil {\delta y^2 \over 2}N \rceil}{2}$.  
Since we are interested in the values of $\delta \ll 1$, then
\begin{equation} \label{ineq:uv}
(1-\delta)^2 {N^2 y^2 \over 2} ~\leq~ v=\binom{N y-\lceil {\delta y^2 \over 2}N \rceil}{2} ~\leq~ u=\binom{N y}{2} \leq {N^2 y^2 \over 2}.
\end{equation}

\medskip
\noindent
Exponential Markov inequality with exponent $s>0$ implies \\
\\
$P\left(N\big[\eta_{(N)}(t)-\eta_{(N)}(t+\delta) \big]<{\delta \over N}v-(1+\delta)N^{2/3}~~\Big| ~\begin{array}{c}N\big[\eta_{(N)}(t)-\eta_{(N)}(t+\delta) \big] \leq {\delta \over N} u+N^{2/3} \\ \eta_{(N)}(t)=y\end{array}~\right)$
\begin{eqnarray*}
\qquad 
& \leq & {e^{-s\,\delta} \over \left(1-{s\,N \over v-N^{5/3}}\right)^{{\delta \over N} v-(1+\delta)N^{2/3}}}\\
& \leq & \exp\left\{-{s\,\delta}+ \Big({\delta \over N} v-(1+\delta)N^{2/3}\Big)
\left({s\,N \over v-N^{5/3}}+{s^2\,N^2 \over (v-N^{5/3})^2} \right) \right\}\\
& \leq & \exp\left\{\left({1 \over 1-N^{5/3}/v}-1\right){s\,\delta}-{s(1+\delta)N^{5/3} 
\over v-N^{5/3}}+ \Big({\delta \over N} v-(1+\delta)N^{2/3}\Big){s^2\,N^2 \over (v-N^{5/3})^2} \right\}\\
& \leq & \exp\left\{{s\,\delta\,N^{5/3}/v\over 1-N^{5/3}/v}-{s(1+\delta)N^{5/3} \over v}+ 
{\delta\,v\,s^2\,N \over (v-N^{5/3})^2} \right\}\\
\end{eqnarray*}
as $~-x-x^2 < \ln(1-x)$  for $x \in \left(0,{1 \over 2}\right)$.
Take $s=N^{1/2}$ to obtain \\
\\
$P\left(N\big[\eta_{(N)}(t)-\eta_{(N)}(t+\delta) \big]<{\delta \over N}v-(1+\delta)N^{2/3}~~\Big| ~\begin{array}{c}N\big[\eta_{(N)}(t)-\eta_{(N)}(t+\delta) \big] \leq {\delta \over N} u+N^{2/3} \\ \eta_{(N)}(t)=y\end{array}~\right)$
\begin{eqnarray} \label{ineq:lbound}
\qquad
& = & \exp\left\{{\delta N^{13/6}/v \over 1-N^{5/3}/v}-{(1+\delta)N^{13/6} \over v}+ 
{\delta v N^2 \over  (v-N^{5/3})^2} \right\} \nonumber \\
& \leq & \exp\left\{{2\delta N^{1/6} \over (1-\delta)^2 y^2 -2N^{-1/3}}-
{2(1+\delta)N^{1/6} \over y^2}+ {2 \delta y^2 \over  
\left((1-\delta)^2 y^2 -2N^{-1/3}\right)^2} \right\} \nonumber \\
& \leq & \exp\left\{{2 N^{1/6}\over y^2} 
\left[{\delta \over (1-\delta)^2 -2N^{-1/3}/y^2}-
(1+\delta)\right]+ 
{3 \delta y^2 \over  (1-\delta)^4 y^4} \right\} \nonumber \\
& \leq & \exp\left\{-{N^{1/6}\over y^2} + 
{3 \delta \over  (1-\delta)^4 y^2} \right\} \nonumber \\
& \leq & \exp\left\{-N^{1/6}+{4\delta \over \epsilon_0^2} \right\}
\end{eqnarray}
for $N$ large enough, by using (\ref{ineq:uv}).

\medskip
\noindent
Thus, multiplying the probabilities of complement events in (\ref{ineq:ubound}) and (\ref{ineq:lbound}) we obtain\\
\\
$P\left( {\delta \over N^2}v-(1+\delta)N^{-1/3} \leq \eta_{(N)}(t)-\eta_{(N)}(t+\delta) \leq {\delta \over N^2} u+N^{-1/3}~~\Big| ~\eta_{(N)}(t)=y~\right)$
$$\geq \left(1-\exp\left\{-N^{1/6}+{4\delta \over \epsilon_0^2} \right\} \right)^2 $$
for any given $t \geq 0$ and $y \in {1 \over N} \mathbb{Z} \cap [\epsilon_0,1]$.

\bigskip
$\bullet$ \textit{Step III.} 
Now, as we already pointed out in (\ref{ineq:uv}), 
\[~(1-\delta)^2 {N^2 \eta^2_{(N)}(t) \over 2} ~\leq~ v~\leq~u~\leq~{N^2 \eta^2_{(N)}(t) \over 2}.\] 
Hence,\\
\\
$P\left( ~\left|{\eta^2_{(N)}(t) \over 2} +\Delta_\delta \eta_{(N)}(t) \right| ~\leq ~\delta+(\delta^{-1}+1)N^{-1/3}~~\Big| ~\eta_{(N)}(t)=y~\right)$
\\
$\geq P\left( (1-\delta)^2 {\eta^2_{(N)}(t) \over 2}-(\delta^{-1}+1)N^{-1/3} \leq -\Delta_\delta \eta_{(N)}(t) \leq {\eta^2_{(N)}(t) \over 2}+\delta^{-1}N^{-1/3}~~\Big| ~\eta_{(N)}(t)=y~\right)$\\
$\geq P\left( {\delta \over N^2}v-(1+\delta)N^{-1/3} \leq \eta_{(N)}(t)-\eta_{(N)}(t+\delta) \leq {\delta \over N^2} u+N^{-1/3}~~\Big| ~\eta_{(N)}(t)=y~\right)$
\begin{equation} \label{ineq:probdiff}
\geq \left(1-\exp\left\{-N^{1/6}+{4\delta \over \epsilon_0^2} \right\} \right)^2 
\end{equation}
for $N$ large enough, where  $~\Delta_\delta f(x):={f(x+\delta)-f(x) \over \delta}~$ denotes the forward difference.
The first inequality above uses the fact that
\[(1-\delta)^2\,\frac{\eta_{(N)}^2(t)}{2}>\frac{\eta_{(N)}^2(t)}{2}-\delta.\]
This is equivalent to
\[(-2+\delta)\,\frac{\eta_{(N)}^2(t)}{2}>-1,\]
which is always true since $\eta_{(N)}(t)\le 1$ and $\delta>0$. 

\bigskip
$\bullet$ \textit{Step IV.} For $K>0$, consider an interval $[0,K]$ partitioned into $M$ subintervals 
$$[t_0,t_1],~[t_1,t_2], ~\hdots, ~[t_{M-1},t_M]$$
of equal length $\delta=K/M$, where $t_0=0$ and $t_M=K$. 

\medskip
\noindent
Let $\epsilon_0=\eta(K)/2=1/(2+K)$, where $\eta(t)=2/(2+t)$ is the solution to the equation (\ref{Aeta_t}) with the initial 
condition $\eta(0)=1$. Consider the following difference equation
\begin{equation}\label{psiE}
\Delta_\delta \psi_{(N)}(t_i)=-{\psi_{(N)}^2(t_i) \over 2}+\mathcal{E}'(t_i)
\end{equation}
with initial condition $\psi_{(N)}(0)=1$, where the error $~|\mathcal{E}'(t_i)| \leq \delta+(\delta^{-1}+1)N^{-1/3}$.  

\medskip
\noindent
{\bf Claim 1.} If $M$ is large enough, then the following is true as we take $N$ large enough.
For any natural number $j \leq M$, if function $\psi_{(N)}(t_i)$ satisfies (\ref{psiE}) 
for all $i \in \{0,1,\hdots,j-1\}$, then
$$\psi_{(N)}(t_j) \geq \epsilon_0.$$
 
\medskip
\noindent 
Indeed, if we take $N \geq M^6$, then 
$$|\mathcal{E}'(t_i)| \leq \delta+(\delta^{-1}+1)N^{-1/3} \leq K/M+1/(KM)+1/M^2.$$

\medskip
\noindent 
Now, since  $\eta(t)=2/(2+t)$ is the solution to the equation (\ref{Aeta_t}) with the initial 
condition $\eta(0)=1$, $\eta(t)$ will satisfy
$$\Delta_\delta \eta(t_i)=-{\eta^2(t_i) \over 2}+\mathcal{E}(t_i)$$
for all $i \in \{0,1,\hdots,M-1\}$, where $~\mathcal{E}(t_i)={\eta''(c_i) \over 2}\delta={\eta^3(c_i) \over 4} \delta$ for some $c_i \in (t_i,t_{i+1})$. 
Hence, as $\eta(t) \leq 1$ for all $t \geq 0$, $~|\mathcal{E}(t_i)| \leq {1 \over 4} \delta$.

\bigskip
\noindent 
Consider the error quantities $~\varepsilon_i:=\psi_{(N)}(t_i)-\eta (t_i)$. 
We have
\begin{eqnarray*}
\varepsilon_{i+1} & = & \psi_{(N)}(t_{i+1})-\eta (t_{i+1})\\
& = & \left[\psi_{(N)}(t_i)-{\psi^2_{(N)}(t_i) \over 2}\delta+\mathcal{E}'(t_i)\delta\right]-\left[\eta(t_i)-{\eta^2(t_i) \over 2}\delta+\mathcal{E}(t_i)\delta \right]\\
& = & \left[\eta(t_i)+\varepsilon_i-{\Big(\eta(t_i)+\varepsilon_i \Big)^2 \over 2}\delta+\mathcal{E}'(t_i)\delta\right]-\left[\eta(t_i)-{\eta^2(t_i) \over 2}\delta+\mathcal{E}(t_i)\delta \right]\\
& = & (1-\eta(t_i)\delta) \varepsilon_i -{\varepsilon_i^2 \over 2}\delta+\delta\Big(\mathcal{E}'(t_i)-\mathcal{E}(t_i)\Big),
\end{eqnarray*}
where $~\Big|\mathcal{E}'(t_i)-\mathcal{E}(t_i)\Big| \leq {5 \over 4}K/M+1/(KM)+1/M^2 < C_K/M~$ 
if $M>1$, with $C_K= {5 \over 4}K+\frac{1}{K}+1$. 
Since $\eta(t_i)>\eta(K)$ for all $i \in \{0,1,\hdots,M-1\}$,
$$|\varepsilon_{i+1}| \leq (1-\eta(K)K/M)|\varepsilon_i|+{\varepsilon_i^2 \over 2}K/M+KC_K/M^2.$$
Taking $M$ large enough so that $KC_K/M < 2\eta(K)$, we can prove by induction that
\begin{equation} \label{ineq:err1}
|\varepsilon_i| \leq iKC_K/M^2.
\end{equation}
Indeed, $\varepsilon_0=0$, and if $|\varepsilon_i| \leq iKC_K/M^2$, then
\begin{eqnarray*}
|\varepsilon_{i+1}| & \leq & (1-\eta(K)K/M)|\varepsilon_i|+{\varepsilon_i^2 \over 2}K/M+KC_K/M^2 \\ 
& = & |\varepsilon_i|+\big(|\varepsilon_i|-2\eta(K)\big)|\varepsilon_i|K/(2M)+KC_K/M^2 \\
& \leq & |\varepsilon_i|+\big(iKC_K/M^2-2\eta(K)\big)|\varepsilon_i|K/(2M)+KC_K/M^2 \\
& \leq & |\varepsilon_i|+KC_K/M^2 \\
& \leq & (i+1)KC_K/M^2,
\end{eqnarray*}
which completes the induction step. \\
The inequality (\ref{ineq:err1}) is therefore valid for all $i \in \{0,\hdots,M-1 \}$, implying
\begin{equation} \label{ineq:err2}
|\varepsilon_i| \leq MKC_K/M^2=\frac{{5 \over 4}K^2+K+1}{M} < \epsilon_0
\end{equation}
for $M$ large enough.

\bigskip
\noindent 
Recall that $\epsilon_0=\eta(K)/2=1/(2+K)$. Then, by (\ref{ineq:err2}),
$$\psi_{(N)}(t_j)=\eta (t_j)+\varepsilon_j \geq \eta(K) -\epsilon_0 = \epsilon_0$$
for all $j \in \{0,1,\hdots,M-1\}$. This proves the above Claim 1.

\bigskip
$\bullet$ \textit{Step V.} 
Consider events
\begin{equation}\label{defAi}
A_i=\left\{~\Delta_\delta \eta_{(N)}(t_i)=-{\eta^2_{(N)}(t_i) \over 2}+\mathcal{E}'(t_i) ~\text{ and } ~|\mathcal{E}'(t_i)| ~\leq \delta+(\delta^{-1}+1)N^{-1/3}~\right\}
\end{equation}
for all $i \in \{0,1,\hdots,M-1\}$. 
Then inequality (\ref{ineq:probdiff})  rewrites as
$$P\left( ~A_j~\Big| ~\eta_{(N)}(t_j)=y~\right)\geq 
\left(1-\exp\left\{-N^{1/6}+{4\delta \over \epsilon_0^2} \right\} \right)^2$$
for any $y \in {1 \over N} \mathbb{Z} \cap [\epsilon_0,1]$.

\medskip
\noindent
Claim 1 implies that $\bigcap\limits_{i=0}^{j-1}A_i$ is contained in the event $\{~\eta_{(N)}(t_j)\in [\epsilon_0,1]~\}$, and therefore
\begin{eqnarray*}
P\left( ~A_j~\Big| ~\bigcap\limits_{i=0}^{j-1}A_i~\right) & = & \sum\limits_{y:~y \in {1 \over N} \mathbb{Z} \cap [\epsilon_0,1] }P\left( ~A_j~\Big| ~\eta_{(N)}(t_j)=y,~\bigcap\limits_{i=0}^{j-1}A_i~\right) P\left(\eta_{(N)}(t_j)=y~\Big| ~\bigcap\limits_{i=0}^{j-1}A_i~\right)\\
& = & \sum\limits_{y:~y \in {1 \over N} \mathbb{Z} \cap [\epsilon_0,1] }P\left( ~A_j~\Big| ~\eta_{(N)}(t_j)=y~\right) P\left(\eta_{(N)}(t_j)=y~\Big| ~\bigcap\limits_{i=0}^{j-1}A_i~\right)\\
& \geq & \left(1-\exp\left\{-N^{1/6}+{4\delta \over \epsilon_0^2} \right\} \right)^2 .
\end{eqnarray*}
as $\sum\limits_{y:~y \in {1 \over N} \mathbb{Z} \cap [\epsilon_0,1] }P\left(\eta_{(N)}(t_j)=y~\Big| ~\bigcap\limits_{i=0}^{j-1}A_i~\right)=P\left(\eta_{(N)}(t_j)\in [\epsilon_0,1]~\Big| ~\bigcap\limits_{i=0}^{j-1}A_i~\right)=1$.
Hence, since we have taken $N \geq M^6$,
\begin{eqnarray}\label{lowerbdAi}
P\left( ~\bigcap\limits_{i=0}^{M-1}A_i~\right) & \geq & 
\left(1-\exp\left\{-N^{1/6}+{4\delta \over \epsilon_0^2} \right\} \right)^{2M} \nonumber \\
& \geq & 
\left(1-\exp\left\{-M+{4K \over \epsilon_0^2 M} \right\} \right)^{2M} \\
& \rightarrow & 1 \quad \text{ as } M \rightarrow \infty. \nonumber
\end{eqnarray}

\medskip
\noindent
We established that with probability greater than 
$P\left( ~\bigcap\limits_{i=0}^{M-1}A_i~\right) \rightarrow 1$ as $M \rightarrow \infty$, 
$\eta_{(N)}(t_i)$ satisfies difference equation (\ref{psiE}) with $\psi_{(N)}(t) \equiv \eta_{(N)}(t)$.

\bigskip
$\bullet$ \textit{Step VI.} Rewriting (\ref{ineq:err2}) for $\psi_{(N)}(t) \equiv \eta_{(N)}(t)$, we see that with probability of at least $P\left( ~\bigcap\limits_{i=0}^{M-1}A_i~\right) \rightarrow 1$,
$$\big| \eta_{(N)}(t_i)-\eta (t_i) \big| ~=|\varepsilon_i| ~< \epsilon_0$$
for all $i \in \{0,1,\hdots,M-1\}$. 
Now, if $t \in (t_i,t_{i+1})$, then
\begin{eqnarray*}
\big| \eta_{(N)}(t)-\eta (t) \big| & \leq & \big| \eta_{(N)}(t) - \eta_{(N)}(t_i)\big|+\big| \eta_{(N)}(t_i)-\eta (t_i) \big|+  \big|\eta (t_i)-\eta (t) \big| \\
& \leq & \big( \eta_{(N)}(t_i) - \eta_{(N)}(t_{i+1})\big)+\left({5 \over 4}K^2+K+1 \right)/M+  \big(\eta (t_i)-\eta (t_{i+1}) \big) \\
& = & \eta_{(N)}(t_i) -\eta (t_i)+\eta (t_{i+1}) - \eta_{(N)}(t_{i+1})+\left({5 \over 4}K^2+K+1 \right)/M+  2\big(\eta (t_i)-\eta (t_{i+1}) \big) \\
& \leq & 3\left({5 \over 4}K^2+K+1 \right)/M+  2\big(\eta (t_i)-\eta (t_{i+1}) \big) \\
& \leq & 3\left({5 \over 4}K^2+K+1 \right)/M+\delta. 
\end{eqnarray*}
as
\begin{equation}\label{MVT}
2\big(\eta (t_i)-\eta (t_{i+1}) \big)=2\delta {d \over dt}\eta(c_i)=\delta\,\eta^2(c_i) \leq \delta  ~~\text{ for some } c_i\in[t_i,t_{i+1}].
\end{equation}

\noindent
Here we used the facts that $ \eta_{(N)}(t)$ and $\eta (t)$ are decreasing 
functions and $\eta(t)=2/(2+t)$ is the solution to Eq. (\ref{Aeta_t}). 
Thus with probability greater than $P\left( ~\bigcap\limits_{i=0}^{M-1}A_i~\right) \rightarrow 1$,
\begin{equation} \label{ineq:KM}
\big\|\eta_{(N)}(t)-\eta (t) \big\|_{L^\infty [0,K]} \leq \left({15 \over 4}K^2+3K+3 \right)/M+K/M
={15 \over 4}K^2/M+4K/M+3/M
\end{equation}
for $M$ large enough and $N \geq M^6$. 

\medskip
\noindent
Therefore, letting $M \rightarrow \infty$, we have shown that
\[\big\|\eta_{(N)}(t)-\eta (t) \big\|_{L^\infty [0,K]} \rightarrow 0 \qquad 
\text{ in probability.}\]

\bigskip
$\bullet$ \textit{Step VII.} 
Take $\epsilon \in (0,1)$ and $\gamma>1$. Let $T_m$ be the time when the first $m=\lfloor(1-\epsilon)N \rfloor$ clusters merge. The expectation for the time $T_m$ is
$$E[T_m]={N \over \binom{N}{2}}+{N \over \binom{N-1}{2}}+\dots+{N \over \binom{N-m+1}{2}}={2m \over N-m}.$$
If we take $K>{2(1-\epsilon) \over \epsilon}\gamma$, then 
$\eta(K) < \eta\left({2(1-\epsilon) \over \epsilon}\gamma\right)< \eta\big(2(1-\epsilon)/\epsilon \big)=\epsilon$, and for any $t \geq K$, $~\big|\eta_{(N)}(t)-\eta (t) \big|>\epsilon~$ implies $~\eta_{(N)}(t)>\epsilon >\eta (t)>0$. 
Thus, by Markov's inequality,\\

\begin{eqnarray}\label{MarkovLast}
P\Big(\big\|\eta_{(N)}(t)-\eta (t) \big\|_{L^\infty [K,\infty)}>\epsilon \Big) & \leq & P\Big(\eta_{(N)}(K)>\epsilon \Big) = P\Big(T_m >K \Big)  \nonumber \\
& \leq & {2(1-\epsilon) \over \epsilon K} < 1/\gamma. 
\end{eqnarray}

\medskip
\noindent
Now, we take $M>\left({15 \over 4}K^2+4K+3 \right)/\epsilon$. 
Then, by (\ref{ineq:KM}),
$$P\Big(\big\|\eta_{(N)}(t)-\eta (t) \big\|_{L^\infty [0,K]}<\epsilon \Big) \geq P\left( ~\bigcap\limits_{i=0}^{M-1}A_i~\right),$$
and 
\begin{eqnarray*}
P\Big(\big\|\eta_{(N)}(t)-\eta (t) \big\|_{L^\infty [0,\infty)}<\epsilon \Big) & \geq & 
P\Big(\big\|\eta_{(N)}(t)-\eta (t) \big\|_{L^\infty [0,K]}<\epsilon \Big)\\
& & \qquad  +P\Big(\big\|\eta_{(N)}(t)-\eta (t) \big\|_{L^\infty [K,\infty)}<\epsilon \Big) -1 \\
& \geq & P\left( ~\bigcap\limits_{i=0}^{M-1}A_i~\right) +(1-1/\gamma) -1\\
& \rightarrow & 1-1/\gamma
\end{eqnarray*}
as we let $~M \rightarrow \infty$. Hence,
$$\limsup\limits_{N \rightarrow \infty}P\Big(\big\|\eta_{(N)}(t)-\eta (t) \big\|_{L^\infty [0,\infty)}<\epsilon \Big) \geq 1-1/\gamma$$
for any given $\gamma>1$. Thus
$$\lim\limits_{N \rightarrow \infty}P\Big(\big\|\eta_{(N)}(t)-\eta (t) \big\|_{L^\infty [0,\infty)}<\epsilon \Big)=1.$$

\medskip
\noindent
Therefore we have shown that $~\|\eta_{(N)}(t)-\eta(t)  \|_{L^\infty [0,\infty)} \rightarrow 0~$ in probability.
\end{proof}

\section{Proof of Lemma \ref{lem3}}
\label{AA3}

\begin{proof}
$\bullet$ \textit{Step I.} 
We will use the setting from the proof of Lemma \ref{lem2}. 
Fix $K>0$ and consider an interval $[0,K]$ partitioned into $M$ subintervals 
$$[t_0,t_1],~[t_1,t_2], ~\hdots, ~[t_{M-1},t_M]$$
of equal length $\delta=K/M$, where $t_0=0$ and $t_M=K$. 
Let $\epsilon_0=\eta(K)/2=1/(2+K)$.

\medskip
\noindent
Once again, let $\eta_{(N)}(t)$ denote the relative total number of clusters. For $i=0,1,\hdots,M-1$, the total number of coalescences within the 
interval $[t_i,t_{i+1}]$ equals $N\big[\eta_{(N)}(t_i)-\eta_{(N)}(t_{i+1})\big]$.
Take $N>M^6$. The probability of the event $~\bigcap\limits_{i=0}^{M-1}A_i$, where $A_i$ was defined in (\ref{defAi}),
was bounded below in (\ref{lowerbdAi}) as follows\\
$P\left( \left| N\big[\eta_{(N)}(t_i)-\eta_{(N)}(t_{i+1})\big]-\delta N {\eta^2_{(N)}(t_i) \over 2} \right|~\leq \delta^2 N+(1+\delta)N^{2/3} \qquad \forall  i=0,1,\hdots,M-1 \right)$ 
$$=P\Big(\bigcap\limits_{i=0}^{M-1}A_i~\Big) ~\geq ~\left(1-\exp\left\{-M+{4K \over \epsilon_0^2 M} \right\} \right)^{2M} \rightarrow  1$$
as  $~M \rightarrow \infty$. 
Recall also that 
$~P\Big(\min\limits_{t \in [0,K]} \eta_{(N)}(t) > \epsilon_0 ~\Big|~\bigcap\limits_{i=0}^{M-1}A_i~\Big)=1$.

\medskip
\noindent
Recall $\eta_{k,N}(t)$ is the number of clusters corresponding to branches of Horton-Strahler order  $k$ at time $t$ relative to the system size $N$, and let $~g_{k,N}(t):=\eta_{(N)}(t)-\sum\limits_{j:~j<k} \eta_{j,N}(t)$. For any $m_i>0$ consider a conditional probability measure $P_{i,m_i}$ where we condition on 
$~\bigcap\limits_{i'=0}^{i-1}A_{i'}~$ and
the values of functions $\{\eta_{j,N}(t_i)\}_{j=0,1,\hdots}$ such that 
$~\eta_{(N)}(t_i)=\sum\limits_{j=0}^\infty \eta_{j,N}(t_i)~$ satisfies
\begin{equation} \label{ineq:mi}
\left| m_i-\delta N {\eta^2_{(N)}(t_i) \over 2} \right|~\leq \delta^2 N+(1+\delta)N^{2/3}.
\end{equation}
Let $E_{i,m_i}$ denote the corresponding conditional expectation. 
Consider the following events:
$$B_{m_i,t_i}= \Big\{\text{ inequality (\ref{ineq:mi}) is satisfied} \Big\},$$
$$D_{m_i,t_i}= \Big\{N\big[\eta_{(N)}(t_i)-\eta_{(N)}(t_{i+1})\big]=m_i \Big\}.$$

\medskip
\noindent
We observe that
\begin{equation}\label{subevent}
A_i=\bigcup\limits_{m_i} \Big[B_{m_i,t_i} \cap D_{m_i,t_i} \Big],
\end{equation}
so $P_{i,m_i}$ is a conditional probability, where we condition on a subevent of  $\bigcap\limits_{i'=0}^{i}A_{i'}$.

\medskip
\noindent
For any $k \in \mathbb{N}^+$ we can represent the coalescences that involve the 
clusters of order $k$ within $[t_i,t_{i+1}]$ as 
$$\eta_{k,N}(t_{i+1})-\eta_{k,N}(t_i)=\xi_1+\xi_2+\hdots+\xi_{m_i},$$
where $~\xi_1,\xi_2,\hdots ,\xi_{m_i}~$ are random variables that correspond to 
the $m_i$ coalescences (of any Horton-Strahler order) within $[t_i,t_{i+1}]$ in 
the order of occurrence.
Here, each $\xi_r$ can take values in $\frac{1}{N}\{-2,-1,0,1\}$; and their dependence 
on $k$ is omitted to simplify the notations.
By construction, the distribution of $\xi_r$ for $1\le r\le m_i$ is completely determined 
by the history $~\mathcal{T}_{r-1}$
of the preceding $r-1$ transitions.
Specifically,

\begin{enumerate}
  \item A transition that decreases $\eta_{k,N}(t)$ by $2/N$ has probability
  $$p_l(-2)~\leq  P_{i,m_i}\left( \xi_r =-2/N ~\Big|~D_{m_i,t_i},~\mathcal{T}_{r-1}  \right) \leq p_u(-2),$$
where $$p_l(-2):=\begin{cases} {\binom{N \eta_{k,N}(t_i)-2\,m_i}{2}/\binom{N\eta_{(N)}(t_i)}{2}} & \text{ if } N \eta_{k,N}(t_i)-2\,m_i \geq 2\\
0 & \text{ otherwise }\end{cases}, $$ 
and $~p_u(-2):={\binom{N \eta_{k,N}(t_i)}{2}/ \binom{N\eta_{(N)}(t_i)-m_i}{2}}$.\\

  \item A transition that increases $\eta_{k,N}(t)$ by $1/N$ has probability
  $$p_l(1)~\leq  P_{i,m_i}\left( \xi_r =1/N ~\Big|~D_{m_i,t_i},~\mathcal{T}_{r-1}  \right) \leq p_u(1),$$
where 
$$p_l(1):=\begin{cases} {\binom{N \eta_{k-1,N}(t_i)-2\,m_i}{2}/\binom{N\eta_{(N)}(t_i)}{2}} & \text{ if } N \eta_{k-1,N}(t_i)-2\,m_i \geq 2\\
0 & \text{ otherwise }\end{cases}, $$ 
and $~p_u(1):={\binom{N \eta_{k-1,N}(t_i)}{2}/ \binom{N\eta_{(N)}(t_i)-m_i}{2}}$ if $k>1$, and if $k=1$, we let $p_l(1)=p_u(1)=0$.\\

  \item A transition that decreases $\eta_{k,N}(t)$ by $1/N$ has probability
  $$p_l(-1) ~\leq  P_{i,m_i}\left( \xi_r =-1/N ~\Big|~ D_{m_i,t_i},~\mathcal{T}_{r-1}  \right) ~\leq p_u(-1)$$
where  
$~p_l(-1):={\max\{(N \eta_{k,N}(t_i)-2\,m_i),0\} N g_{k+1,N}(t_i)/ \binom{N\eta_{(N)}(t_i)}{2}}~$ 
and\\
\\
$p_u(-1):={N^2 \eta_{k,N}(t_i) g_{k+1,N}(t_i)/ \binom{N\eta_{(N)}(t_i)-m_i}{2}}$. 
\end{enumerate}

\medskip
\noindent
Next, let $~p(-2):= \eta^2_{k,N}(t_i) / \eta^2_{(N)}(t_i)$, 
\qquad $p(1):=\begin{cases}
    \eta^2_{k-1,N}(t_i) / \eta^2_{(N)}(t_i) & \text{ if } k>1 \\
   0   & \text{ if } k=1
\end{cases}$,\\
\\
$p(-1):=2\eta_{k,N}(t_i)g_{k+1,N}(t_i)/  \eta^2_{(N)}(t_i)$, \quad $~p(0):=1-p(-2)-p(-1)-p(1)$,
and $\xi$ be a random variable with the values $\{-2,-1,0,1\}$ specified by the probabilities 
$\{p(-2),p(-1),p(0),p(1)\}$.
Also let $~\xi^+=\xi \cdot {\bf 1}_{\xi>0}~$ and $~\xi^-=\xi \cdot {\bf 1}_{\xi<0}$. 

\medskip
\noindent
Observe that since we conditioned on a sub-event of  $\bigcap\limits_{i'=0}^{i}A_{i'}$, then $\eta_{(N)}(t_i) \geq \epsilon_0$ and therefore
$$p_l(-2)=p(-2)+\mathcal{O}(\delta) \quad \text{ and }\quad p_u(-2)=p(-2)+\mathcal{O}(\delta),$$
$$p_l(1)=p(1)+\mathcal{O}(\delta) \quad \text{ and }\quad p_u(1)=p(1)+\mathcal{O}(\delta),$$
$$p_l(-1)=p(-1)+\mathcal{O}(\delta) \quad \text{ and }\quad p_u(-1)=p(-1)+\mathcal{O}(\delta).$$

\medskip
\noindent
Let $~\xi_r^+=\xi_r \cdot {\bf 1}_{\xi_r>0}~$ and $~\xi_r^-=\xi_r \cdot {\bf 1}_{\xi_r<0}$. Then
$$\eta_{k,N}(t_{i+1})-\eta_{k,N}(t_i)=X_+ +X_-,$$
where
$$X_+=\xi_1^++\xi_2^++\hdots+\xi_{m_i}^+$$
and
$$X_-=\xi_1^-+\xi_2^-+\hdots+\xi_{m_i}^-.$$

\medskip
\noindent
Next, for any $\lambda^+, \lambda^- \geq 0$ and $s \in [0,1]$ consider\\
$$E_{i,m_i}\left[e^{sN\big[\lambda^+ X_+ +\lambda^- X_-\big]}~\Big|~D_{m_i,t_i} \right]=\prod\limits_{r=1}^{m_i} E_{i,m_i}\left[e^{sN[\lambda^+\xi_r^+ +\lambda^-\xi_r^-]} ~\Big| ~ D_{m_i,t_i},\mathcal{T}_{r-1} \right],$$
where for all $r$,
\begin{eqnarray*}
\lefteqn{E_{i,m_i}\left[e^{sN[\lambda^+\xi_r^+ +\lambda^-\xi_r^-]} ~\Big| ~ D_{m_i,t_i},\mathcal{T}_{r-1} \right]}\\ 
& \leq & e^{-2\lambda^- s}p_u(-2)+e^{-\lambda^- s}p_u(-1)+e^{\lambda^+ s}p_u(1)+(1-p_l(-2)-p_l(-1)-p_l(1))\\
& \leq & e^{-2\lambda^- s}p(-2)+e^{-\lambda^- s}p(-1)+e^{\lambda^+ s}p(1)+p(0)+C\delta\\
&=& E\left[e^{s\, [\lambda^+\xi^+ +\lambda^-\xi^-]} \right]+C\delta
\end{eqnarray*}
for large enough $C>0$. Hence,
$$E_{i,m_i}\left[e^{sN\big[\lambda^+ X_++\lambda^-X_-\big]}~\Big|~D_{m_i,t_i} \right]
 \leq \Big(E\left[e^{s\, [\lambda^+\xi^+ +\lambda^-\xi^-]} \right]+C\delta \Big)^{m_i}.$$

\medskip
\noindent
Therefore, by the exponential Markov inequality with the exponent $s$, for any $m_i$ such that (\ref{ineq:mi}) is satisfied,
\begin{eqnarray*}
\lefteqn{P_{i,m_i}\Big(N\big[\lambda^+ X_++\lambda^-X_- \big] \geq E[\lambda^+\xi^+ +\lambda^-\xi^-]\,m_i + m_i^{14/15} ~\Big| ~D_{m_i,t_i}~ \Big)}\\
& ~~~\leq & E_{i,m_i}\left[e^{sN\big[\lambda^+ X_++\lambda^-X_-\big]}~\Big|~D_{m_i,t_i} \right] e^{-s\big(E[\lambda^+\xi^+ +\lambda^-\xi^-]\,m_i  + m_i^{14/15}\big) } \\
& ~~~\leq  & \Big(E\left[e^{s\,[\lambda^+\xi^+ +\lambda^-\xi^-]}\right]+C\delta \Big)^{m_i} e^{-s\big(E[\lambda^+\xi^+ +\lambda^-\xi^-]\,m_i  + m_i^{14/15}\big) }  \\
&~~~=& \Big(E\left[e^{s(\lambda^+[\xi^+-E[\xi^+]]+\lambda^-[\xi^- -E[\xi^-]])}\right]+e^{-sE[\lambda^+\xi^+ +\lambda^-\xi^-]}C\delta\Big)^{m_i}e^{-sm_i^{14/15}}\\
&~~~=& \Big( 1+E\left[s\left(\lambda^+[\xi^+-E[\xi^+]]+\lambda^-[\xi^- -E[\xi^-]]\right)\right]+C\delta +\mathcal{O}(s^2+s\delta)\Big)^{m_i} e^{-sm_i^{14/15}}\\
& ~~~= & \Big(1+C\delta+\mathcal{O}(s^2+s\delta) \Big)^{m_i} e^{-sm_i^{14/15} } \\
& ~~~\leq & \exp\left\{m_i \big[C\delta+\mathcal{O}(s^2+s\delta) \big]-sm_i^{14/15} \right\},
\quad {\rm as~}s,\delta\to 0.
\end{eqnarray*}

\medskip
\noindent
Next, taking $2M^6>N>M^6$ and $M$ large enough, and plugging $s=2C\delta m_i^{1/15}=\mathcal{O}(M^{-2/3})$ (as $M \rightarrow \infty$) into the above exponential Markov inequality, we obtain
\begin{eqnarray}\label{lambdalower}
\lefteqn{P_{i,m_i}\Big(N\big[\lambda^+ X_++\lambda^-X_- \big]  \geq E[\lambda^+\xi^+ +\lambda^-\xi^-]\,m_i +  m_i^{14/15} ~\Big| ~D_{m_i,t_i}~ \Big)} \nonumber \\
& ~~~\leq & \exp\Big\{-C\delta m_i+\mathcal{O}(M^{11/3})\Big\} \qquad \qquad \qquad \qquad \nonumber \\
& ~~~\leq & \exp\Big\{-A M^4\Big\}
\end{eqnarray}
for sufficiently small positive $A <CK^2\epsilon_0^2/2\leq CK^2\eta^2_{(N)}(t_i)/2$ and sufficiently large $M$ as $m_i$ satisfies (\ref{ineq:mi}), e.g. let $A=CK^2\epsilon_0^2/10$.

\medskip
\noindent
The exponential in $M^4$ lower bound on 
$$P_{i,m_i}\Big(N\big[\lambda^+ X_++\lambda^-X_- \big]  \leq E[\lambda^+\xi^+ +\lambda^-\xi^-]\,m_i  - m_i^{14/15} ~\Big| ~D_{m_i,t_i}~ \Big)$$
follows via a symmetrical argument. 
Specifically, for $C>0$ large enough, and all $s \in [0,1]$,
\[E_{i,m_i}\left[e^{-sN\big[\lambda^+ X_++\lambda^-X_- \big]}~\Big|~D_{m_i,t_i} \right] \leq \Big(E\left[e^{-s\,[\lambda^+\xi^+ +\lambda^-\xi^-]}\right]+C\delta \Big)^{m_i}.\]
Therefore, taking $s=2C\delta m_i^{1/15}=\mathcal{O}(M^{-2/3})$, we obtain

\begin{eqnarray}\label{lambdaupper}
\lefteqn{
P_{i,m_i}\Big(N\big[\lambda^+ X_++\lambda^-X_- \big] \leq E[\lambda^+\xi^+ +\lambda^-\xi^-]\,m_i - m_i^{14/15} ~\Big| ~D_{m_i,t_i}~ \Big)} \nonumber \\
& ~~~\leq & E_{i,m_i}\left[e^{-sN\big[\lambda^+ X_++\lambda^-X_- \big]}~\Big|~D_{m_i,t_i} \right] 
e^{s\big(E[\lambda^+\xi^+ +\lambda^-\xi^-]\,m_i - m_i^{14/15}\big) } \nonumber \\
& ~~~\leq  & \Big(E\left[e^{-s\,[\lambda^+\xi^+ +\lambda^-\xi^-]}\right]+C\delta \Big)^{m_i} e^{s\big(E[\lambda^+\xi^+ +\lambda^-\xi^-]\,m_i  - m_i^{14/15}\big) }  \nonumber \\
& ~~~= & \Big(1+C\delta+\mathcal{O}(s^2+s\delta) \Big)^{m_i} e^{-sm_i^{14/15} } \nonumber \\
& ~~~\leq & \exp\left\{m_i \big(C\delta+\mathcal{O}(s^2+s\delta) \big)-sm_i^{14/15} \right\} \nonumber \\
& ~~~\leq & \exp\Big\{-C\delta m_i+\mathcal{O}(M^{11/3})\Big\} \nonumber \\
& ~~~\leq & \exp\Big\{-A M^4\Big\} 
\end{eqnarray}
\\
for sufficiently small positive $A <CK^2\epsilon_0^2/2\leq CK^2\eta^2_{(N)}(t_i)/2$ 
and sufficiently large $M$.

\bigskip
\noindent
Thus, plugging $\lambda^+=\lambda^-=1$ into (\ref{lambdalower}) and (\ref{lambdaupper}), we obtain the following inequality. For each $k$ and $M$ large enough, there exists $a>0$ such that
{\footnotesize
$$P_{i,m_i}\Big( \Big|\big(\eta_{k,N}(t_{i+1})-\eta_{k,N}(t_i) \big) - E[\xi]\,m_i/N \Big| < m_i^{14/15}/N ~\Big| ~D_{m_i,t_i}~ \Big) \geq  1-\exp\Big\{-a M^4 \Big\}$$
}
for all $i=0,1,\hdots,M-1$ and $m_i$ satisfying (\ref{ineq:mi}).

\medskip
\noindent
Now, (\ref{subevent}) implies for any event $F$ dependent on $\{\eta_{j,N}(t_i)\}_j$ and $\{\eta_{j,N}(t_{i+1})\}_j$, \
$$P\Big(F \Big|~ \bigcap\limits_{i'=0}^{i}A_{i'}~\Big)=\sum\limits_{m_i, ~\{\eta_{j,N}(t_i)\}_j}P_{i,m_i}\Big(F \Big| ~D_{m_i,t_i}~ \Big)P\Big(~B_{m_i,t_i} \cap D_{m_i,t_i} ~\Big| ~ \bigcap\limits_{i'=0}^{i}A_{i'} \Big),$$
where
$$\sum_{m_i, ~\{\eta_{j,N}(t_i)\}_j} P\Big(~B_{m_i,t_i} \cap D_{m_i,t_i}~\Big| ~ \bigcap\limits_{i'=0}^{i}A_{i'} \Big)=1.$$

Therefore, since here $m_i^{14/15}/N=\mathcal{O}(M^{-4/3})$, $\delta^2=\mathcal{O}(M^{-2})$, and $(1+\delta)N^{-1/3}=\mathcal{O}(M^{-2})$, there is a large enough   $c_k >0$ such that

$$P\Big( \Big|\big[\eta_{k,N}(t_{i+1})-\eta_{k,N}(t_i) \big] - 
E[\xi]\delta {\eta^2_{(N)}(t_i) \over 2} \Big| 
< c_k \delta^{4/3} ~\Big| ~\bigcap\limits_{i'=0}^{i}A_{i'}~ \Big)$$
{\footnotesize
$$=\sum\limits_{m_i, ~\{\eta_{j,N}(t_i)\}_j}P_{i,m_i}\Big(\Big|
\big[\eta_{k,N}(t_{i+1})-\eta_{k,N}(t_i) \big] - E[\xi]\delta {\eta^2_{(N)}(t_i) \over 2} \Big| < c_k \delta^{4/3} \Big| ~D_{m_i,t_i}~ \Big)$$ 
$$\qquad \qquad \qquad \times P\Big(~B_{m_i,t_i} \cap D_{m_i,t_i}~\Big| ~ \bigcap\limits_{i'=0}^{i}A_{i'} \Big)$$
$$\geq \sum\limits_{m_i, ~\{\eta_{j,N}(t_i)\}_j}P_{i,m_i}\Big(\Big|
\big[\eta_{k,N}(t_{i+1})-\eta_{k,N}(t_i) \big] - E[\xi]m_i/N \Big| < m_i^{14/15}/N  \Big| ~D_{m_i,t_i}~ \Big)$$ 
$$\qquad \qquad \qquad \times P\Big(~B_{m_i,t_i} \cap D_{m_i,t_i}~\Big| ~ \bigcap\limits_{i'=0}^{i}A_{i'} \Big)$$
}
\begin{equation} \label{ineq:aMb}
 \geq  1-\exp\Big\{-a M^4\Big\}
\end{equation}
for all $i=0,1,\hdots,M-1$, $~2M^6>N>M^6$, and $M$ large enough, as $\eta_{(N)}(t_i)\in [\epsilon_0,1]$ for all $i$.

\bigskip
\noindent
$\bullet$ \textit{Step II.} 
We obtain the following system of difference equations with the initial conditions and the error bound as mentioned below. 
\begin{eqnarray}\label{eqn:hydr}
\Delta_\delta \eta_{1,N}(t_i) & = & -\eta_{1,N}(t_i)\eta_{(N)}(t_i) +\mathcal{E}'_1(t_i) \nonumber \\
& & \\
\Delta_\delta \eta_{k,N}(t_i) & = & {\eta^2_{k-1,N}(t_i) \over 2}  -\eta_{k,N}(t_i) g_{k,N}(t_i) +\mathcal{E}'_{k}(t_i)  \quad \text{ for } k \geq 2\nonumber 
\end{eqnarray}
with the initial conditions
 $$\Big(\eta_{1,N}(0), ~\eta_{2,N}(0), ~\hdots, ~\eta_{k,N}(0), ~\hdots \Big)=(1,0,0,\hdots),$$
where for a given $\rho\in \mathbb{N}$ and $c=\max\limits_{1\le k\le \rho}\{c_k\}$ we have $|\mathcal{E}'_k(t_i)| < c \delta^{1/3}$ for each $1\le k \le \rho$.
Here, for each $k$, the $k$-th equation holds with probability of at least\\
\\
$1-\sum\limits_{i=0}^M \Big[1-P\Big(~\bigcap\limits_{i'=0}^{i}A_{i'}~ \Big) \cdot \left(1-\exp\Big\{-a M^{2/3}\Big\}\right)\Big] \geq 1-M\Big[1-P\Big(~\bigcap\limits_{i'=0}^{M-1}A_{i'}~ \Big) \cdot \left(1-\exp\Big\{-a M^{2/3}\Big\}\right)\Big]$
$$\geq 1-M\left[1-\left(1-\exp\left\{-M+{4K \over \epsilon_0^2 M} \right\} \right)^{2M} 
\left(1-\exp\Big\{-a M^{2/3}\Big\}\right)\right]$$  
$$\geq 1+M\exp\Big\{-a M^{2/3}\Big\}-M\left[1-\left(1-\exp\left\{-M+{4K \over \epsilon_0^2 M} \right\} \right)^{2M} \right]$$  
$$\rightarrow 1 \text{ as } M \rightarrow \infty.$$

\bigskip
\noindent
Finally, the same error propagation analysis as in Step IV in the proof of Lemma \ref{lem2} is applied to compare the above 
difference equations (\ref{eqn:hydr}) to the difference equations 
that correspond to the following system of ODEs
\begin{eqnarray*}
{d \over dt} \eta_1(t) & = & -\eta_1(t) \eta(t) \\
& & \\
{d \over dt} \eta_{k}(t) & = & {\eta^2_{k-1}(t) \over 2}  -\eta_{k}(t) g_{k}(t) \quad \text{ for } k \geq 2 \\
\end{eqnarray*}
with the initial conditions
 $$\Big(\eta_1(0), ~\eta_2(0), ~\hdots, ~\eta_k(0), ~\hdots \Big)=(1,0,0,\hdots),$$
where $~g_k(t):=\eta(t)-\sum\limits_{i:~i<k} \eta_i(t)$. The above system of ODEs can be converted into the following system of difference equations
\begin{eqnarray}\label{eqn:hydr2}
\Delta_\delta \eta_1(t_i) & = & -\eta_1(t_i) \eta(t_i) +\mathcal{E}_1(t_i) \nonumber \\
& & \\
\Delta_\delta \eta_k(t_i) & = & {\eta^2_{k-1}(t_i) \over 2}  -\eta_k(t_i) g_k(t_i) +\mathcal{E}_{k}(t_i) \quad \text{ for } k \geq 2 \nonumber 
\end{eqnarray}
with the error
$$\mathcal{E}_{k}(t_i)={\eta''_k(c_{i,k}) \over 2}\delta \qquad \text{ for some } c_{i,k} \in (t_i,t_{i+1}).$$

\medskip
\noindent
Here $|\mathcal{E}_{1}(t_i)|={|\eta''_1(c_{i,1})| \over 2}\delta<{3 \over 4}\delta~$
as
$~\eta''_1(t) =-\big[\eta_1(t) \eta(t)\big]'={3 \over 2}\eta_1(t) \eta^2(t)$.

\medskip
\noindent
The error for $k>1$ is
$$|\mathcal{E}_{k}(t_i)|={|\eta''_k(c_{i,k})| \over 2}\delta \leq {k+2 \over 2}\delta$$
as
$$\eta''_k(t)=\left[{\eta^2_{k-1}(t) \over 2}  -\eta_{k}(t) g_{k}(t)\right]'=\eta_{k-1}(t)\eta'_{k-1}(t)-\eta'_{k}(t) g_{k}(t)-\eta_{k}(t) g'_{k}(t)$$
$$=\eta_{k-1}(t)\left({\eta^2_{k-2}(t) \over 2}  -\eta_{k-1}(t) g_{k-1}(t)\right)-\left({\eta^2_{k-1}(t) \over 2}  -\eta_{k}(t) g_{k}(t)\right) g_{k}(t)$$
$$-\eta_{k}(t) \left(-{\eta^2_k(t) \over 2}-\eta_1(t) \eta(t)+\sum\limits_{i:~2 \leq i<k} \left[{\eta^2_{i-1}(t) \over 2}  -\eta_{i}(t) g_{i}(t)\right]   \right)$$
and for each $i$, $~|\eta_{i}(t)|\leq 1~$ and  $~|g_{i}(t)| \leq 1$.

\bigskip
\noindent
$\bullet$ \textit{Step III.} 
Next, the error propagates as in (\ref{ineq:err2}), iteratively producing for each $k \in \mathbb{N}^+$
$$\varepsilon_{k,i}:=\eta_{k,N}(t_i)-\eta_k(t_i)=\mathcal{O}(M^{-1}).$$

\medskip
\noindent
Indeed, if $\varepsilon_i=\eta_{(N)}(t_i)-\eta(t_i)$, then conditioning on the event $~\bigcap\limits_{i=0}^{M-1}A_i$,  the approximation error $~\varepsilon_i$ was shown to satisfy $~|\varepsilon_i| \leq iKC_K /M^2$.

\medskip
\noindent
Let $d_1:={3 \over 4}$, and for $k>1$, $d_k:={k+2 \over 2}$. Then $|\mathcal{E}_{k}(t_i)| \leq d_k \delta$. Next let $~\varepsilon_{0,i}:=0$ for all $i$. Also, we observe that $~\varepsilon_{k,0}=0$ for all $k \geq 0$ because of the same initial conditions in systems (\ref{eqn:hydr}) and (\ref{eqn:hydr2}).

\medskip
\noindent
From the difference equations (\ref{eqn:hydr}) and (\ref{eqn:hydr2}), we have the error propagating as follows
\begin{eqnarray*}
\varepsilon_{k,i+1} & = & \varepsilon_{k,i}+ \delta \left( {\eta^2_{k-1,N}(t_i) \over 2}-{\eta^2_{k-1}(t_i) \over 2} \right) -\delta \Big(\eta_{k,N}(t_i) g_{k,N}(t_i)-\eta_k(t_i) g_k(t_i) \Big)\\
& & +\delta\big(\mathcal{E}'_{k}(t_i)-\mathcal{E}_{k}(t_i) \big)\\
& = & \varepsilon_{k,i}+ \delta \left( \eta_{k-1}(t_i)\varepsilon_{k-1,i}+{\varepsilon^2_{k-1,i} \over 2} \right) \\ 
& & -\delta \left((\eta_{k}(t_i)+\varepsilon_{k,i}) \left[\varepsilon_i -\sum\limits_{k'=1}^{k-1}\varepsilon_{k',i} \right] +g_k(t_i)\varepsilon_{k,i} \right)\\
& & +\delta\big(\mathcal{E}'_{k}(t_i)-\mathcal{E}_{k}(t_i) \big)
\end{eqnarray*}
and therefore
\begin{eqnarray}\label{propagate2}
|\varepsilon_{k,i+1}| & \leq & |\varepsilon_{k,i}|+\delta |\varepsilon_{k-1,i}| + \delta {\varepsilon^2_{k-1,i} \over 2} + \delta \left[|\varepsilon_i| + \sum\limits_{k'=1}^k|\varepsilon_{k',i} |\right] \nonumber\\
& &+ \delta \varepsilon_{k,i} \left[|\varepsilon_i| + \sum\limits_{k'=1}^{k-1}|\varepsilon_{k',i} |\right] 
+c\delta^{4/3}+\delta^2 d_k.
\end{eqnarray}

\noindent 
The inequality (\ref{propagate2}) is crucial for proving the following statement by induction. We claim that for each integer $\rho>0$ and $M$ large enough,
 $$|\varepsilon_{k,i}| \leq (c+1)2^k{\delta^{1/3} \over \rho}\big[(1+2\delta \rho)^i-1 \big] , \quad \text{ for all } k \in \{1,\hdots,\rho\} \text{ and } ~ i=0,1,\hdots,M-1.$$ 
 
\noindent
The basis step follows from the initial conditions $~\varepsilon_{0,i}=0~$ and $~\varepsilon_{k,0}=0$. The inductive step is obtained from (\ref{propagate2}) as follows. Suppose for a choice of $k \in \{1,\hdots,\rho\}$ and $i$, $$|\varepsilon_{k',j}| \leq (c+1)2^{k'}{\delta^{1/3} \over \rho}\big[(1+2\delta \rho)^j-1 \big]$$ for all $j=0,1,\hdots,M-1$ whenever $k' <k$, and $$|\varepsilon_{k,j}| \leq (c+1)2^k{\delta^{1/3} \over \rho}\big[(1+2\delta \rho)^j-1 \big]$$ whenever $j \leq i$.

Observe that 
\[\delta\,\sum\limits_{k'=1}^k|\varepsilon_{k',i}|\leq
\delta \rho (c+1)2^k{\delta^{1/3} \over \rho}\big[(1+2\delta \rho)^i-1 \big]\]
and hence
\begin{eqnarray*} 
\lefteqn{|\varepsilon_{k,i}|+\delta |\varepsilon_{k-1,i}|+
\delta\,\sum\limits_{k'=1}^k|\varepsilon_{k',i}|+c\,\delta^{4/3}}\nonumber\\
& \leq &
(c+1)2^k{\delta^{1/3} \over \rho}\big[(1+\delta/2+\delta \rho)(1+2\delta \rho)^i-1 \big]-C_1\delta^{4/3}\nonumber\\
& \leq &
(c+1)2^k{\delta^{1/3} \over \rho}\big[(1+2\delta \rho)^{i+1}-1\big]-C_1\delta^{4/3},
\end{eqnarray*}
with $C_1=(c+1)2^{k-1}\,\rho^{-1} +(c+1)\,2^k-c >0$.
At the same time, all other terms in (\ref{propagate2}) are estimated from above 
by functions that have higher powers of $\delta$:
\begin{eqnarray*}
\delta {\varepsilon^2_{k-1,i} \over 2} & \leq &
(c+1)^2 2^{2k-3}{\delta^{5/3} \over \rho^2}\big[e^{2K\rho}-1 \big]^2,\nonumber\\
\delta\,|\varepsilon_i|&\leq&\delta^2\,C_K,\nonumber\\
\delta\,\varepsilon_{k,i}\,|\varepsilon_i|&\leq &
C_K(c+1)2^k{\delta^{7/3} \over \rho}\big[e^{2K\rho}-1 \big],\nonumber\\
\delta\,\varepsilon_{k,i}\,\sum\limits_{k'=1}^{k-1}|\varepsilon_{k',i} |&\leq&
(c+1)^2 2^{2k}{\delta^{5/3} \over \rho}\big[e^{2K\rho}-1 \big]^2,
\end{eqnarray*}
where we used the observation
$~(1+2\delta \rho)^i \leq (1+2\delta \rho)^M \leq e^{2K\rho}$.
This implies that
\[|\varepsilon_{k,i+1}|\leq(c+1)2^k{\delta^{1/3} \over \rho}\big[(1+2\delta \rho)^{i+1}-1\big]\]
for $M$ large enough, and therefore $\delta$ small enough, thus proving the claim. 
Hence
 $$|\varepsilon_{k,i}| \leq (c+1)2^k{\delta^{1/3} \over \rho}\big[(1+2\delta \rho)^i-1 \big] \leq (c+1)2^k{\delta^{1/3} \over \rho}\big[e^{2K\rho}-1 \big]=\mathcal{O}(\delta^{1/3})$$
for any $\rho$ and all $k \in \{1,\hdots,\rho\}$.

\medskip
\noindent
Therefore, 
conditioning on the event $~\bigcap\limits_{i=0}^{M-1}A_i$, we have the following upper bound for any $k \in \{1,\hdots,\rho\}$ and for all $i \in \{0,1,\hdots,M-1\}$. 
If $t \in (t_i,t_{i+1})$, then
\begin{eqnarray*}
\big| \eta_{k,N}(t)-\eta_k(t) \big| & \leq & \big| \eta_{k,N}(t) - \eta_{k,N}(t_i)\big|+\big| \eta_{k,N}(t_i)-\eta_k(t_i) \big|+  \big|\eta_k (t_i)-\eta_k (t) \big| \\
& \leq & 2\big( \eta_{(N)}(t_i) - \eta_{(N)}(t)\big)+ (c+1)2^k{\delta^{1/3} \over \rho}\big[e^{2K\rho}-1 \big]+  \big|\eta_k (t_i)-\eta_k (t) \big| \\
& \leq & 2\big( \eta_{(N)}(t_i) - \eta_{(N)}(t_{i+1})\big)+ (c+1)2^k{\delta^{1/3} \over \rho}\big[e^{2K\rho}-1 \big]+ \big|\eta_k (t_i)-\eta_k (t) \big| \\
& = & 2\big(\eta_{(N)}(t_i) -\eta (t_i)\big)+2\big(\eta (t_{i+1}) - \eta_{(N)}(t_{i+1})\big)+ 2\big(\eta (t_i)-\eta (t_{i+1}) \big)\\
& & \qquad \qquad \qquad \qquad \qquad  +(c+1)2^k{\delta^{1/3} \over \rho}\big[e^{2K\rho}-1 \big]+ \big|\eta_k (t_i)-\eta_k (t) \big|\\
& \leq & \left(5K^2+4K+4 \right)/M+(c+1)2^k{\delta^{1/3} \over \rho}\big[e^{2K\rho}-1 \big]+3\delta. 
\end{eqnarray*}
as the net change $\big| \eta_{k,N}(t) - \eta_{k,N}(t_i)\big|$ in the number of clusters of order $k$ is dominated by twice the net change $\eta_{(N)}(t_i) - \eta_{(N)}(t)$ in the total number of clusters. We also used 
$$\eta_{(N)}(t_{i'}) -\eta (t_{i'}) \leq \left({5 \over 4}K^2+K+1 \right)/M \quad \text{ for all } i' \in \{0,1,\hdots,M\}$$
shown in (\ref{ineq:err2}),
$$2\big(\eta (t_i)-\eta (t_{i+1}) \big) \leq \delta$$
shown in (\ref{MVT}), and that there exists $c'_i \in (t_i,t_{i+1})$ such that
$$ \big|\eta_k (t_i)-\eta_k (t) \big|=(t-t_i)\left|{d \over dt}\eta_k(c'_i)\right|=(t-t_i)\left|{\eta^2_{k-1}(c'_i) \over 2}  -\eta_{k}(c'_i) g_{k}(c'_i)\right| \leq 2\delta.$$

\medskip
\noindent
Thus, for any $k$,
$$\|\eta_{k,N}-\eta_k  \|_{L^\infty [0,K]} \rightarrow 0 \quad\text{in probability}.$$

\bigskip
\noindent
$\bullet$ \textit{Step IV.} Finally, observe that  for any $\epsilon>0$ and for $K>2$ large enough so that $~\eta(K) <\epsilon$,
$$\eta_k (t) \leq \eta(t) \leq \eta(K) <\epsilon \text{ for all } t \geq K$$
and, by (\ref{MarkovLast}), 
\begin{eqnarray*}
P\Big(\big\|\eta_{k,N}(t)-\eta_k (t) \big\|_{L^\infty [K,\infty)}>\epsilon \Big) & \leq & P\Big(\big\|\eta_{k,N}(t) \big\|_{L^\infty [K,\infty)}>\epsilon \Big)\\
& \leq & P\Big(\big\|\eta_{(N)}(t) \big\|_{L^\infty [K,\infty)}>\epsilon \Big)\\
& = & P\Big(\eta_{(N)}(K)>\epsilon \Big) \\
& \leq & {2(1-\epsilon) \over \epsilon K}.
\end{eqnarray*}
Thus, together with the previous step,  we have shown that for each $k$, $$\|\eta_{k,N}-\eta_k  \|_{L^\infty [0,\infty)} \rightarrow 0$$ in probability.
\end{proof}

\section{Proof of Lemma \ref{lem1}}
\label{AA1}

\begin{proof}
Observe that when we  plug in $\lambda^+=1$ and $\lambda^-=0$ into (\ref{lambdalower}) and (\ref{lambdaupper}), we obtain that in the difference equations (\ref{eqn:hydr}), the number of emerging clusters of Horton-Strahler order $j$ within the time interval $[t_i,t_{i+1}]$ divided by $N$ is
$${p(1)m_i+\mathcal{O}(m_i^{14/15}) \over N} ={\eta^2_{j-1,N}(t_i) \over 2}\cdot \delta +\mathcal{O}(\delta^{4/3})$$
for all $i=0,1,\hdots,M-1$, $\delta=K/M$, and $m_i$ satisfying (\ref{ineq:mi}), with probability approaching 1 exponentially fast as $2N>M^6>N \rightarrow \infty$.
Here $\sum\limits_{i=0}^{{K \over \delta}-1} {\eta^2_{j-1,N}(t_i) \over 2}\cdot \delta$ converges almost surely to $\int\limits_0^K {\eta^2_{j-1,N}(t) \over 2} dt$ as $\delta \rightarrow 0$.

Hence, for $j \geq 2$, the total number $N_j(K)$ of emerging clusters of  Horton-Strahler order $j$ within the time interval $[0,K]$ divided by $N$ is
$$N_j(K)/N=\int\limits_0^K {\eta^2_{j-1,N}(t) \over 2} dt+\mathcal{O}(\delta^{1/3})$$
with probability approaching $1$ as $M \rightarrow \infty$.

\bigskip
\noindent
Fix $\varepsilon>0$. We established that $~\|\eta_{j,N}-\eta_j  \|_{L^\infty [0,K]} \rightarrow 0~$ in probability. Then 
$$\left|\int\limits_0^K {\eta^2_{j-1}(t) \over 2} dt-\int\limits_0^K {\eta^2_{j-1,N}(t) \over 2} dt \right| 
\leq {K \over 2}\|\eta_{j-1}+\eta_{j-1,N}\|_{L^\infty [0,K]} \cdot \|\eta_{j-1}-\eta_{j-1,N}\|_{L^\infty [0,K]}  \rightarrow 0.$$
Thus,$~\left|N_j(K)/N -\int\limits_0^K {\eta^2_{j-1}(t) \over 2} dt \right| <\varepsilon$ with probability $\mathcal{P}_{K,\varepsilon, N} \rightarrow 1$ as $N\to\infty$. 

\medskip
\noindent
Now, for $K>2(1-\varepsilon)/\varepsilon$, 
$$\int\limits_K^\infty {\eta^2_{j-1}(t) \over 2} dt \leq \int\limits_K^{\infty} {\eta^2(t) \over 2} ~dt=\int\limits_K^{\infty} {2 \over (t+2)^2} ~dt={2 \over K+2} < \varepsilon$$
and
$$P\Big(\eta_{(N)}(K)<\varepsilon \Big) ~\geq 1-{2(1-\varepsilon) \over \varepsilon K}.$$

\medskip
\noindent
Therefore, the total number of emerging clusters of  Horton-Strahler order $j$ within $[0,\infty)$ time interval  divided by $N$ satisfies
\begin{eqnarray*}
\lefteqn{P\left( \left|N_j/N-\int\limits_0^\infty {\eta^2_{j-1}(t) \over 2} dt\right|<3\varepsilon \right)}\\
& ~~~\geq & P\left( \Big(N_j-N_j(K)\Big)/N<\varepsilon ,~~~  \left|N_j(K)/N-\int\limits_0^K {\eta^2_{j-1}(t) \over 2} dt \right|<\varepsilon  \right)\\
& ~~~\geq & \min\left\{P\left( \Big(N_j-N_j(K)\Big)/N<\varepsilon \right), \quad  P\left(\left|N_j(K)/N-\int\limits_0^K {\eta^2_{j-1}(t) \over 2} dt \right|<\varepsilon  \right)  \right\}\\
& ~~~\geq & \min\left\{1-{2(1-\varepsilon) \over \varepsilon K}, \quad  \mathcal{P}_{K,\varepsilon, N} \right\}\\ 
& ~~~\rightarrow & 1-{2(1-\varepsilon) \over \varepsilon K}
\end{eqnarray*}
as $N \rightarrow \infty$.

\medskip
\noindent
Thus, since we can take $K$ as large as we want,
$$P\left( \left|N_j/N-\int\limits_0^\infty {\eta^2_{j-1}(t) \over 2} dt\right|<3\varepsilon \right) \rightarrow 1.$$ 

\end{proof}

\bibliographystyle{amsplain}

\newpage

\begin{figure}[p] 
\centering\includegraphics[width=.4\textwidth]{HS_example.jpg}
\caption[Example of Horton-Strahler indexing]
{Example of Horton-Strahler ordering.
Two order-2 branches are depicted by heavy lines.
The branch to the left from the root consists of one vertex;
the branch to the right from the root consists of two vertices.}
\label{fig_HST}
\end{figure}

\begin{figure}[p] 
\centering\includegraphics[width=0.7\textwidth]{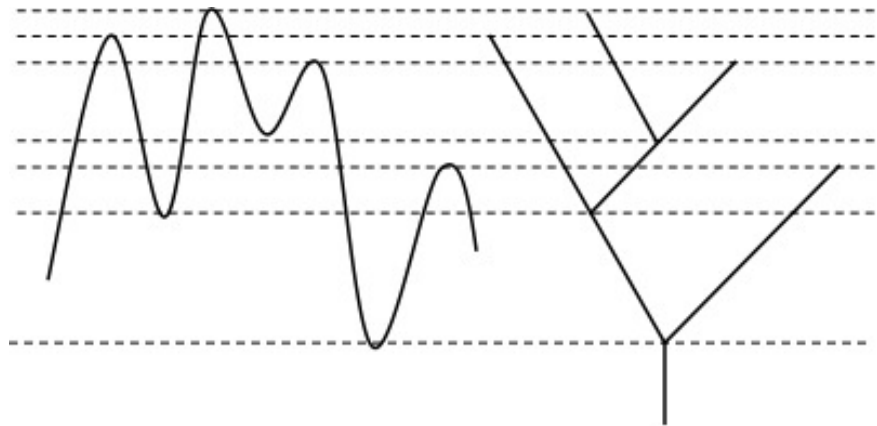}
\caption{Function $X_t$ (panel a) with a finite number of local
extrema and its level set tree $\textsc{level}(X)$ (panel b).
}
\label{fig3}
\end{figure}

%
%
%

\end{document}